\documentclass{amsart}
\usepackage{amssymb}
\usepackage{color}
\usepackage{amsxtra} 
\usepackage{mathrsfs} 
\usepackage{yfonts}

\newtheorem{theorem}{Theorem}%[section]
\newtheorem{lem}[theorem]{Lemma}
\newtheorem{prop}[theorem]{Proposition}
 
\newtheorem*{theorem*}{Claim} 
\newtheorem*{theorem**}{Theorem}

\theoremstyle{definition}

\theoremstyle{remark}
\newtheorem{remark}[theorem]{Remark} 
\newtheorem*{remark*}{Remark}

\numberwithin{equation}{section}
\numberwithin{theorem}{section} 

\begin{document}

\title[The solution map of the Euler equations]
{Continuity of the solution map of the Euler equations in H\"older spaces and weak norm inflation in Besov spaces} 

\author{Gerard Misio{\l}ek}
\address{Department of Mathematics, University of Notre Dame, IN 46556, USA} 
%\curraddr{Department of Mathematics, University of Colorado, Boulder, CO 80309-0395, USA} 
\email{gmisiole@nd.edu} 

\author{Tsuyoshi Yoneda}
\address{Graduate School of Mathematical Sciences, University of Tokyo, Komaba 3-8-1 Meguro, Tokyo 153-8914, Japan} 
\email{yoneda@ms.u-tokyo.ac.jp}
%\thanks{The second author was supported in part by NSF Grant \#000000.}

%    General info

\date{\today} 



\begin{abstract} 
We construct an example showing that the solution map of the Euler equations is not continuous 
in the H\"older space from $C^{1,\alpha}$ to $L^\infty_tC^{1,\alpha}_x$ for any $0<\alpha<1$. 
On the other hand we show that it is continuous when restricted to the little H\"older subspace $c^{1,\alpha}$. 
We apply the latter to prove 
an ill-posedness result for solutions of the vorticity equations in Besov spaces 
near the critical space $B^1_{2,1}$. 
As a consequence we show that a sequence of best constants of the Sobolev embedding theorem 
near the critical function space is not continuous. 
\end{abstract} 

\maketitle

%%%%%%%%%%%%%%%%%%%%%%%%%%%%%%%%
\section{Introduction} 
\label{sec:Intro} 

We study the Cauchy problem for the Euler equations of an incompressible and inviscid fluid 
\begin{align} \nonumber 
&u_t + u\cdot\nabla u = - \nabla p, 
\qquad 
t \geq 0, \; x \in \mathbb{R}^n 
\\ 
\label{E} 
&\mathrm{div}\, u = 0 
\\ 
\nonumber 
&u(0) = u_0 
\end{align} 
where $u=u(t,x)$ and $p=p(t,x)$ denote the velocity field and the pressure function of the fluid respectively. 
The first rigorous results for \eqref{E} were proved in the framework of H\"older spaces by 
Gyunter \cite{Gu}, Lichtenstein \cite{Li} and Wolibner \cite{Wo}. 
More refined results using a similar functional setting were obtained subsequently by 
Kato \cite{Ka}, Swann \cite{Sw}, Bardos and Frisch \cite{BF}, Ebin \cite{Eb}, Chemin \cite{Ch}, 
Constantin \cite{Co1} and Majda and Bertozzi \cite{MB} among others. 
The main focus in these papers was on existence and uniqueness of $C^{1,\alpha}$ solutions 
and the question of continuity with respect to initial conditions was not explicitly addressed. 
Recall that, according to the definition of Hadamard, a Cauchy problem is said to be 
locally well-posed in a Banach space $X$ if for any initial data in $X$ there exists a unique solution 
which persists for some $T>0$ in the space $C([0,T), X)$ and which depends continuously on the data. 
Otherwise the problem is said to be ill-posed. 
It was pointed out by Kato \cite{KaKdV} that this notion of well-posedness is rather strong and 
may not be suitable for certain problems studied in the literature. Instead, it is frequently required that 
the solution persist in a larger space such as $L^\infty([0,T), X)$ or $C_w([0,T), X)$ (the subscript $w$ 
indicates weak continuity in the time variable). 

Systematic studies of ill-posedness of the Cauchy problem \eqref{E} are of a more recent date 
and concern a wide range of phenomena including 
gradual loss of regularity of the solution map, energy dissipation and non-uniqueness of weak solutions, 
see e.g. Yudovich \cite{Yu2}, Koch \cite{Ko}, Morgulis, Shnirelman and Yudovich \cite{MSY}, Eyink \cite{Ey}, 
Constantin, E and Titi \cite{CET} or Shnirelman \cite{Sh2}. 
Recently, Bardos and Titi \cite{BT} found examples of solutions in H\"older spaces $C^\alpha$ 
and the Zygmund space $B^1_{\infty,\infty}$ which exhibit an instantaneous loss of smoothness 
in the spatial variable for any $0<\alpha<1$. 
Similar examples in logarithmic Lipschitz spaces $\mathrm{logLip}^\alpha$ were given by 
the authors in \cite{MY}. 
In another direction Cheskidov and Shvydkoy \cite{CS} constructed periodic solutions that are discontinuous 
in time at $t=0$ in the Besov spaces $B^s_{p, \infty}$ where $s>0$ and $2<p\leq\infty$. 
In particular, it follows that the Cauchy problem \eqref{E} is not well-posed in the sense of Hadamard 
in $C([0,T), B^s_{p,\infty})$ although it is known that the corresponding solution map 
is well-defined in $L^\infty([0,T),B^s_{p,\infty})$, see for instance \cite{BCD}, Chap. 7.  
More recently, in a series of papers Bourgain and Li \cite{BL, BL1} constructed smooth solutions 
which exhibit 
instantaneous blowup 
in borderline spaces such as $W^{n/p+1,p}$ for any $1 \leq p < \infty$ and 
$B^{n/p+1}_{p,q}$ for any $1 \leq p < \infty$ and $1 < q \leq \infty$ 
as well as in the standard spaces $C^k$ and $C^{k-1,1}$ for any integer $k \geq 1$; 
see also Elgindi and Masmoudi \cite{ElMa} and \cite{MY2}. 
As observed in \cite{BL1} the cases $C^k$ and $C^{k-1,1}$ are particularly intriguing in view of 
the classical existence and uniqueness results mentioned above. 

One of our goals in this paper is to revisit the picture of local well-posedness in the sense of Hadamard 
for the Euler equations in H\"older spaces. 
We present a simple example based on a DiPerna-Majda type shear flow which shows that in general 
the data-to-solution map of \eqref{E} is not continuous into the space $L^\infty([0,T),C^{1,\alpha})$ 
for any $0<\alpha<1$. 
On the other hand, we show that continuity of this map is restored (in the strong sense) if the Cauchy problem 
is restricted to the so called little H\"older space $c^{1,\alpha}$. 
The failure of continuity in our example does not seem to be related to the mechanism 
described in \cite{BL1} which essentially relies on unboundedness of the double Riesz transform in $L^\infty$. 
Rather, it can be explained by the fact that smooth functions are not dense 
in the standard $C^{1,\alpha}$ spaces. 
This phenomenon should be compared with the results of \cite{CS} where,
however, as mentioned above, the ill-posedness mechanism is different 
and with those of \cite{HM} where it is shown that the solution map cannot be uniformly continuous 
in the Sobolev space $H^s$ with $s>0$.
We point out that continuity of the solution map for the Euler equations in Sobolev spaces 
$W^{s,p}$ for $p\geq 2$ and $s > 2/p + 2$ is of course well known 
(see e.g., Ebin and Marsden \cite{EbMa}, Kato and Lai \cite{KL} or Kato and Ponce \cite{KP-Duke}, 
see also Appendix.
% Remark \ref{commutator inequalities}). 
However, we could not find the corresponding result for $c^{1,\alpha}$ in the literature 
although it should be familiar to the experts in the field. 
\begin{theorem} \label{LIP} 
The solution map of the incompressible Euler equations \eqref{E} is not continuous as a map 
from 
$C^{1,\alpha}(\mathbb{R}^3)$ 
to 
$L^\infty([0,T), C^{1,\alpha}(\mathbb{R}^3))$ 
for any $0<\alpha<1$. 
\end{theorem}
\begin{theorem} \label{LWP} 
The incompressible Euler equations \eqref{E} are locally well-posed 
in the sense of Hadamard in the little H\"older space 
$c^{1,\alpha}(\mathbb{R}^n)$ 
for any $0<\alpha<1$ and $n=2$ or $3$. 
\end{theorem} 
In particular, Theorem \ref{LWP} implies that the solution map is continuous from bounded subsets of 
$c^{1,\alpha}(\mathbb{R}^n)$ to $C([0,T), c^{1,\alpha}(\mathbb{R}^n))$. 
\begin{remark}
 Although in the present paper we are primarily concerned with the local-in-time problem, a few comments on global well-posedness are in order.
In the light of Theorem \ref{LWP} it is natural to expect that Wolibner's global result can be reformulated 
as a global Hadamard well-posedness result in $c^{1,\alpha}$. 
In fact, the proof in Section \ref{Holder-wp} below shows that if $u_0$ is in $c^{1,\alpha}$ 
then the corresponding particle trajectories $\eta(t)$ will retain their $c^{1,\alpha}$ regularity 
for any finite interval of time (they will remain in the set $\mathscr{U}_\delta$ for some suitably chosen $\delta >0$, 
cf. \eqref{U}-\eqref{eq:jac} below) 
provided that the flow can be continued in $C^{1,\alpha}$. 
The latter however is guaranteed by the Beale-Kato-Majda criterion (see e.g. \cite[Theorem 4.3]{MB}) 
because vorticity is conserved along particle trajectories in 2D. 
\end{remark}
%\newline 
%\indent 
\begin{remark}
The proof of Theorem \ref{LIP} is based on a local property of $C^{1,\alpha}$ to construct a counterexample. 
It would be also interesting to find an explicit counterexample in the Besov space framework 
$B^{1+\alpha}_{\infty,\infty}$ such that $\inf_{\ell\in\mathbb{Z}_+}2^{(1+\alpha)\ell}\|\hat\psi_\ell\ast u_0\|_{L^\infty}>0$ 
(non-decaying property on the Fourier side).  
\end{remark} 

As an application of Theorem \ref{LWP} we prove 
an ill-posedness result 
for the vorticity equations that involves a family of Besov spaces. Although this result is weaker than 
instantaneous blowup 
described by Bourgain and Li our methods can be applied in the borderline end-point spaces such as 
$B^2_{2,1}(\mathbb{R}^2)$ which lie just outside the range of the spaces considered in \cite{BL}. 
Recall that existence and uniqueness results for \eqref{E} in $B^2_{2,1}(\mathbb{R}^2)$ 
are already known, cf. e.g. Vishik \cite{V} or Chae \cite{Chae}. 
The proof uses continuity of the data-to-solution map in $c^{1,\alpha}$ 
as well as several technical lemmas proved in our earlier paper \cite{MY2}. 
In this respect the present paper can be viewed as a continuation of \cite{MY2}. 
\begin{theorem}\label{LIP-Besov} 
Let $M_j \nearrow \infty$ be an increasing sequence of positive numbers. 
There exists a sequence of smooth rapidly decaying initial data $\{\tilde{u}_{0,j}\}_{j=1}^\infty$ 
and two sequences of indices $\{r_j\}_{j=1}^\infty$ and $\{q_j\}_{j=1}^\infty$ 
with $r_j \to 2$ and $q_j \to 1$ such that 
\begin{equation*} 
\|\tilde{u}_{0,j}\|_{B^2_{r_j, q_j}} \lesssim 1 
\quad \text{and} \quad 
\|\tilde{u}_j(t)\|_{B^2_{r_j, q_j}} > M_j 
\quad \text{for some} \;\, 0<t<M_j^{-3}. 
\end{equation*} 
\end{theorem} 
\begin{remark} 
It is of interest to compare this result with the a'priori estimates for local solutions 
in Besov norms. If we let $C_j>0$ denote the best constant in the Besov version of 
the Sobolev embedding 
\begin{equation} \label{embedding} 
\|\nabla\tilde{u}_j\|_\infty \leq C_j \|\tilde{u}_j\|_{B^2_{r_j,q_j}} 
\end{equation} 
(cf. Chae \cite[Rem. 2.1]{Chae}) 
then Theorem \ref{LIP-Besov} implies that $C_j \nearrow \infty$ or else we get a contradiction with 
the standard bound 
$
\sup_{0 \leq t \leq T_j}{\|\tilde{u}_j(t)\|_{B^2_{r_j,q_j}}} \leq C_{0,j}\|\tilde{u}_{0,j}\|_{B^2_{r_j,q_j}} 
$ 
where $C_{0,j}$ depends on $C_j$.
On the other hand, for $1\leq p<\infty$ and $1\leq q\leq \infty$, we also have the estimate 
\begin{equation} \label{embedding-limit} 
\|\nabla\tilde{u}_\infty\|_\infty \leq C_\infty \|\tilde{u}_\infty\|_{B^{1+2/p}_{p,q}} 
\end{equation} 
for some finite constant $C_\infty >0$ if and only if $q= 1$ (see \cite[Theorem 11.4, p.170]{T}). 
This suggests that the dependence of the constants in \eqref{embedding} and \eqref{embedding-limit} 
on the Besov parameters is not continuous. 
To the best of our knowledge this dependence has not been investigated in the literature. 
\end{remark} 

In the next section we recall the basic set up and notation. 
In Section \ref{shear flow} we prove Theorem \ref{LIP} by constructing a shear flow counterexample 
in the $C^{1,\alpha}$ space. 
Local Hadamard well-posedness in $c^{1,\alpha}$ is shown in Section \ref{Holder-wp}. 
The proof of Theorem \ref{LIP-Besov} is given in Section \ref{Besov}.

%%%%%%%%%%%%%%%%%
%%%%%%%%%%%%%%%%%%%%%%%%%%%%%%% 
\section{The basic setup: function spaces and diffeomorphisms} 
\label{sec:prelim} 

Let $\psi_0 \in \mathscr{S}(\mathbb{R}^n)$ be any function of Schwartz class satisfying $0 \leq \psi_0 \leq 1$ 
and $\mathrm{supp}\, \psi_0 \subset \{ \xi \in \mathbb{R}^n: 1/2 \leq |\xi| \leq 2 \}$ 
and such that 
\begin{equation*} 
\sum_{l \in \mathbb{Z}} \psi_l(\xi) = 1, 
\quad 
\text{for any} \; \xi \neq 0 
\end{equation*} 
where $\psi_l(\xi) = \psi_0(2^{-l}\xi)$. 
For any $s >0$ and $1 \leq p, q \leq \infty$ let $B^s_{p,q}(\mathbb{R}^n)$ 
denote the inhomogeneous Besov space equipped with the norm 
\begin{equation} \label{Besov-inh} 
\| f \|_{B^s_{p,q}} = \| f \|_{L^p} + \| f \|_{\dot{B}^s_{p,q}} 
\end{equation} 
where the homogeneous semi-norm is given by 
\begin{equation} \label{Besov-h} 
\| f \|_{\dot{B}^s_{p,q}} 
= \left\{ 
\begin{matrix} \displaystyle
~~~\bigg( \sum_{l \in \mathbb{Z}} 2^{slq}\|\widehat{\psi}_l \ast f\|_{L^p}^q \bigg)^{1/q}& 
\mathrm{if} \quad 1\leq q <\infty 
\\  \displaystyle
\sup_{l \in \mathbb{Z}}{ 2^{sl}} \| \widehat{\psi}_l \ast f\|_{L^p}& 
\mathrm{if} \qquad\;\; q=\infty 
\end{matrix} 
\right.
\end{equation} 
for any $f \in \mathscr{S}'(\mathbb{R}^n)$. 
In particular, if $s=k+\alpha$ is not an integer then $B^s_{\infty,\infty}$ is the H\"older space 
$C^{k,\alpha}(\mathbb{R}^n)$ with the standard norm 
\begin{align*} 
\|\varphi\|_{k,\alpha} 
= 
\| \varphi \|_{C^k} 
+ 
[D^k\varphi]_\alpha 
\end{align*} 
where 
$$ 
[D^k\varphi]_\alpha 
= 
\sum_{|\beta|=k} \sup_{x\neq y} \frac{|D^\beta\varphi(x) - D^\beta\varphi(y)|}{|x-y|^\alpha}, 
\qquad\quad 
0 < \alpha < 1, \; k \in \mathbb{N}. 
$$ 
Let $c^{k,\alpha}(\mathbb{R}^n)$ denote the closed subspace of $C^{k,\alpha}(\mathbb{R}^n)$ 
consisting of those functions whose derivatives satisfy the vanishing condition 
\begin{align} \label{nLH} 
\lim_{h \to 0} \sup_{0<|x-y|<h} \frac{|D^\beta\varphi(x) - D^\beta\varphi(y)|}{|x-y|^\alpha} = 0 
\end{align} 
for any multi-index $|\beta|=k$. 
It is well known that $c^{k,\alpha}(\mathbb{R}^n)$ is an interpolation space containing the smooth functions 
as a dense subspace, cf. e.g. \cite{BeLo}. 

In what follows we will use an alternative formulation of the fluid equations 
in terms of particle trajectories and vorticity. 
Any sufficiently smooth velocity field $u$ solving \eqref{E} has a flow which traces out 
a curve $t \to \eta(t,x)$ of diffeomorphisms starting at the identity configuration $e(x)=x$ 
with initial velocity $u_0$. 
%% 
%\begin{equation} \label{eq:2Dvort} 
%\omega = \frac{\partial v_2}{\partial x_1} - \frac{\partial v_1}{\partial x_2} 
%\quad 
%\text{if $n=2$} 
%\end{equation} 
%% 
%and 
%% 
%\begin{equation} \label{eq:3Dvort} 
%\omega 
%= 
%\bigg( 
%\frac{\partial v_3}{\partial x_2} - \frac{\partial v_2}{\partial x_3}, 
%\frac{\partial v_1}{\partial x_3} - \frac{\partial v_3}{\partial x_1}, 
%\frac{\partial v_2}{\partial x_1} - \frac{\partial v_1}{\partial x_2} 
%\bigg) 
%\quad 
%\text{if $n=3$} 
%\end{equation} 
%% 
Using the incompressibility constraint 
$\det{D\eta(t,x)}=1$ 
and the Biot-Savart law the equations satisfied by the flow can be written in the form 
\begin{align} \label{IDE} 
&\frac{d\eta}{dt}(t,x) 
= 
\int_{\mathbb{R}^n} K_n\big(\eta(t,x) - \eta(t,y) \big) \omega(t,\eta(t,y)) \, dy, 
\qquad 
t \geq 0,\; x \in \mathbb{R}^n 
\\ \nonumber 
&\eta(0, x) = x 
\end{align} 
where $\omega=\mathrm{curl}\, u$ is the vorticity\footnote{If $n=2$ we can identify the vorticity of $u$ 
with the function $\omega=\nabla^\perp\cdot u$ and if $n=3$ with the vector field 
$\omega=\nabla \times u$.} 
%satisfying the transport formulas 
%% 
%\begin{equation} \label{eq:2Dvort} 
%\omega(t,\eta(t,x)) = \omega_0(x), 
%\qquad 
%\text{if $n=2$} 
%\end{equation} 
%% 
%and 
%% 
%\begin{equation} \label{eq:3Dvort} 
%\omega(t,\eta(t,x)) = D\eta(t,x) \omega_0(x), 
%\qquad 
%\text{if $n=3$} 
%\end{equation} 
% 
and the kernel $K_n$ is given by 
\begin{align} \label{eq:KBS2} 
K_2(x) = \frac{1}{2\pi} \left( -\frac{x_2}{|x|^2}, \frac{x_1}{|x|^2} \right), 
\qquad 
x \in \mathbb{R}^2 
\end{align} 
and 
\begin{align} \label{eq:KBS3} 
K_3(x)y  = \frac{1}{4\pi} \frac{x \times y}{|x|^3}, 
\qquad 
x, y \in \mathbb{R}^3. 
\end{align} 

For our purposes it will be sufficient to take as a configuration space of the fluid the set of 
those diffeomorphisms of $\mathbb{R}^n$ which differ from the identity by a function of 
class $c^{1,\alpha}$. Let 
\begin{align} \label{U} 
\mathscr{U}_\delta = \Big\{ 
\eta:\mathbb{R}^n \to \mathbb{R}^n: 
\eta = e + \varphi_\eta,~ \varphi_\eta \in c^{1,\alpha}(\mathbb{R}^n) 
~\text{and}~ 
\|\varphi_\eta\|_{1,\alpha} < \delta 
\Big\} 
\end{align} 
where $\delta >0$ is chosen small enough so that 
\begin{equation} \label{eq:jac} 
\inf_{x \in \mathbb{R}^n}{\det D\eta(x)} > \frac{1}{2}. 
\end{equation} 
Clearly, $\mathscr{U}_\delta$ can be identified with an open ball centered at the origin 
in $c^{1,\alpha}(\mathbb{R}^n)$. 
The next two lemmas collect some elementary properties 
of compositions and inversions of diffeomorphisms in $\mathscr{U}_\delta$ that will be used 
in Section \ref{Holder-wp}. 
\begin{lem} \label{lemH} 
Let $0<\alpha<1$. 
Suppose that $\eta$ and $\xi$ are in $\mathscr{U}_\delta$ and $\psi \in c^{1,\alpha}(\mathbb{R}^n)$. 
Then 
$\psi\circ\eta$ and $\psi\circ\eta^{-1}$ are also of class $c^{1,\alpha}$ and we have 
\begin{align} \label{comp1} 
\| \psi \circ \eta \|_{1,\alpha} \lesssim C \|\psi\|_{1,\alpha} 
\quad \text{and} \quad 
\| \psi \circ \eta^{-1} \|_{1,\alpha} \lesssim C \|\psi\|_{1,\alpha} 
\end{align} 
where $C>0$ depends only on $\delta$ and $\alpha$. 
Furthermore, we have $\xi\circ\eta-e$ and $\eta^{-1}-e$ are also of class $c^{1,\alpha}$.
\end{lem} 
\begin{proof} 
First, observe that 
$\| \psi\circ\eta\|_\infty = \|\psi\|_\infty$ 
and 
$\| D(\psi\circ\eta)\|_\infty = \|D\psi\|_\infty \|D\eta\|_\infty$ 
and therefore the first of the inequalities in \eqref{comp1} follows at once from 
\begin{align} \label{comp2} 
[D(\psi\circ\eta)]_\alpha 
%= 
%[D\psi\circ\eta \cdot D\eta]_\alpha 
&\leq 
[D\psi\circ\eta]_\alpha \|D\eta\|_\infty + \|D\psi\circ\eta\|_\infty [D\eta]_\alpha 
\\ \nonumber 
&\leq 
[D\psi]_\alpha \|D\eta\|_\infty^{1+\alpha} + \|D\psi\|_\infty [D\eta]_\alpha 
\end{align} 
and 
$[D\eta]_\alpha = [D\varphi_\eta]_\alpha$ 
where $\eta = e + \varphi_\eta$ with $\|\varphi_\eta\|_{1,\alpha}<\delta$. 
Similarly, we have 
$\|\psi{\circ}\eta^{-1}\|_\infty = \|\psi\|_\infty$ 
and 
from \eqref{eq:jac} and $D\eta^{-1} = (D\eta)^{-1}\circ\eta^{-1}$ we get 
$$ 
\|D(\psi\circ\eta^{-1})\|_\infty \lesssim \|D\psi\|_\infty \|D\eta\|_\infty 
$$ 
and 
\begin{align} \label{comp33} 
[(D\eta)^{-1}]_\alpha 
= 
\big[(\det{D\eta})^{-1} \mathrm{adj}(D\eta) \big]_\alpha 
\lesssim 
(1 + \|D\eta\|_\infty^2) [D\eta]_\alpha 
\end{align} 
which, in turn, with the help of \eqref{comp2} yields 
\begin{align} \label{comp22} 
[D(\psi\circ\eta^{-1})]_\alpha 
\leq 
[D\psi]_\alpha \|D\eta\|_\infty^{1+\alpha} 
+ 
\|D\psi\|_\infty (1+\|D\eta\|_\infty^2) [D\eta]_\alpha. 
\end{align} 
From these bounds we obtain the second of the inequalities in \eqref{comp1}.  

Finally, observe that if $\xi \in \mathscr{U}_\delta$ with $\xi = e + \varphi_\xi$ then 
$\xi\circ\eta = e + \varphi_{\xi\circ\eta}$ 
where 
$\varphi_{\xi\circ\eta} = \varphi_\eta + \varphi_\xi\circ\eta$. 
Therefore, using \eqref{comp1} we get 
\begin{align*} 
\| \varphi_{\xi\circ\eta}\|_{1,\alpha} 
\lesssim 
\|\varphi_\eta\|_{1,\alpha} + \|\varphi_\xi\|_{1,\alpha} 
\end{align*} 
and combining \eqref{comp1} with \eqref{comp2} and the vanishing condition \eqref{nLH} 
we conclude that 
$\varphi_{\xi\circ\eta} \in c^{1,\alpha}(\mathbb{R}^n)$. 
Similarly, we also have $\eta^{-1} = e + \varphi_{\eta^{-1}}$ where 
$\varphi_{\eta^{-1}} = - \varphi_{\eta}\circ\eta^{-1}$. 
Applying the second of the estimates in \eqref{comp1} together with \eqref{comp22} 
and \eqref{nLH} we find again that $\varphi_{\eta^{-1}} \in c^{1,\alpha}(\mathbb{R}^n)$. 
\end{proof} 
\begin{lem} \label{lemHC} 
Let $0<\alpha<1$. 
Suppose that $\eta, \xi$ and $\zeta$ are in $\mathscr{U}_\delta$. Then 
\begin{equation} \label{comp4} 
\| \xi\circ\eta - \zeta\circ\eta\|_{1,\alpha} 
\lesssim 
C \|\varphi_\xi - \varphi_\zeta\|_{1,\alpha} 
\end{equation} 
and for any $\psi \in c^{2,\alpha}(\mathbb{R}^n)$ we have 
\begin{equation} \label{comp3} 
\| \psi\circ\eta - \psi\circ\xi\|_{1,\alpha} 
\lesssim 
C \|\psi\|_{2,\alpha} \|\varphi_\eta - \varphi_\xi\|_{1,\alpha} 
\end{equation} 
where $C>0$ depends only on $\delta$ and $\alpha$. 
Furthermore, the functions $\xi, \eta \to \xi\circ\eta$ and $\eta \to \eta^{-1}$ are continuous 
in the H\"older norm topology. 
\end{lem} 
\begin{proof} 
From the estimates of Lemma \ref{lemH} we obtain as before 
\begin{align*} 
\| \xi\circ\eta - \zeta\circ\eta \|_{1,\alpha} 
%&= 
%\|(\xi - \zeta){\circ}\eta\|_\infty 
%+ 
%\|D(\xi - \zeta){\circ}\eta \, D\eta\|_\infty 
%+ 
%[D(\xi-\zeta){\circ}\eta \, D\eta ]_\alpha 
%\\ 
&\lesssim 
\|\varphi_\xi - \varphi_\zeta\|_\infty + \|D\eta\|_\infty \|D(\varphi_\xi - \varphi_\zeta)\|_\infty 
\\ 
&+ 
[D\eta]_\alpha \|D(\varphi_\xi - \varphi_\zeta)\|_\infty 
+ 
\|D\eta\|_\infty^{1+\alpha} [D(\varphi_\xi - \varphi_\zeta)]_\alpha 
\\ 
&\lesssim 
\Big(1 + \|D\eta\|_\infty + \|D\eta\|_\infty^{1+\alpha} + [D\eta]_\alpha \Big) \| \varphi_\xi - \varphi_\zeta\|_{1,\alpha} 
\end{align*} 
which implies the estimate in \eqref{comp4}. 
On the other hand, using \eqref{comp1} and the algebra property of H\"older functions we have 
\begin{align*} 
\| \psi\circ\eta - \psi\circ\xi \|_{1,\alpha} 
\leq 
\int_0^1 \| D\psi \big( r\eta + (1-r)\xi \big) (\eta - \xi) \|_{1,\alpha} dr 
\lesssim 
\|D\psi\|_{1,\alpha} \| \eta - \xi\|_{1,\alpha} 
\end{align*} 
which gives \eqref{comp3} since $\eta - \xi = \varphi_\eta - \varphi_\xi$. 
From \eqref{comp4} and \eqref{comp3} we conclude that composition of diffeomorphisms 
in $\mathscr{U}_\delta$ is continuous with respect to $\xi$ and $\eta$. 

Finally, using the second of the inequalities in \eqref{comp1} we have 
\begin{align*} 
\| \xi^{-1} - \eta^{-1}\|_{1,\alpha} 
\lesssim 
\| \xi^{-1}\circ\eta - e \|_{1,\alpha}. 
\end{align*} 
By density, given any $\varepsilon >0$ pick a smooth $\zeta : \mathbb{R}^n \to \mathbb{R}^n$ 
such that $\| \zeta - \xi^{-1} \|_{1,\alpha} < \varepsilon$ 
and estimate the above expression further by 
\begin{align*} 
\| \xi^{-1}\circ\eta - \zeta\circ\eta \|_{1,\alpha}
+
\| \zeta\circ\eta - \zeta\circ\xi \|_{1,\alpha}
+
\| \zeta\circ\xi - \xi^{-1}\circ\xi \|_{1,\alpha}. 
\end{align*} 
The first and the third of these terms can be bounded using the first inequality in \eqref{comp1} 
by $C\varepsilon$. For the middle term we use \eqref{comp3} to bound it by 
$\| \zeta \|_{2,\alpha} \| \eta - \xi \|_{1,\alpha}$. 
\end{proof} 
% 

%%%%%%%%%%%%%%%%
%%%%%%%%%%%%%%%%%%%%%%%%%%%
\section{A 3D shear flow in $C^{1,\alpha}$}
\label{shear flow} 

In this section we prove Theorem \ref{LIP} by constructing a $C^{1,\alpha}$ shear flow for which 
the data-to-solution map of \eqref{E} fails to be continuous. Shear flow solutions were introduced in \cite{DM}. 
They were used recently in \cite{BT} to exhibit instantaneous loss of smoothness of the Euler equations 
in $C^\alpha$ for any $0<\alpha<1$. 

\begin{proof}[Proof of Theorem \ref{LIP}] 
Let $t\geq 0$ and consider 
$$ 
u(t,x) = \big( f(x_2), 0, h(x_1 - tf(x_2)) \big) 
\quad \text{and} \quad 
v(t,x) = \big( g(x_2), 0, h(x_1 - tg(x_2)) \big) 
$$ 
where $f$, $g$ and $h$ are bounded real-valued functions of one variable 
of class $C^{1, \alpha}$ with any $0<\alpha<1$. 
It is not difficult to verify that both $u$ and $v$ satisfy the Euler equations with initial conditions 
$$
u_0(x) = \big( f(x_2), 0, h(x_1) \big) 
\quad \text{and} \quad 
v_0(x) = \big( g(x_2), 0, h(x_1) \big). 
$$ 

Given any $\varepsilon>0$ we can arrange so that $f$ and $g$ satisfy 
\begin{equation*}
\| u_0 - v_0 \|_{1,\alpha} = \| f - g \|_{1,\alpha} < \varepsilon 
\end{equation*}
and then choose $h$ such that 
\begin{equation*} 
h'(x_1) = |x_1|^\alpha 
\qquad \quad 
\text{for all} 
\;\; 
-2a \leq x_1 \leq 2a 
\end{equation*} 
where $a=\max \big\{ \|f\|_\infty, \|g\|_\infty \big\}$ 
and assume that $f(x) \neq g(x)$ on $(-a,a)$. 

Next, we estimate the norm of the difference of the corresponding solutions. 
For any $0 < t \leq 1$ we have 
\begin{align*}
\| u(t) - v(t) \|_{1,\alpha} 
&= 
\| f - g \|_{1,\alpha} 
+ 
\big\| h( \cdot - tf(\cdot)) - h( \cdot - tg( \cdot )) \big\|_{1,\alpha} 
\\ 
&\geq 
\big\| \nabla \big( h( \cdot - tf( \cdot )) - h(\cdot - tg(\cdot)) \big) \big\|_{0,\alpha} 
\\
%&\geq 
%\big\| \partial_{x_1} \big(h( \cdot - tf( \cdot )) - h(\cdot - tg(\cdot)) \big) \big\|_{C^{\alpha}} 
%\\ 
&= 
\big\| h'( \cdot - tf(\cdot)) - h'(\cdot - tg(\cdot)) \big\|_{0,\alpha}. 
\end{align*} 
% 
%Let $b=\min \big\{ \|f\|_\infty,\|g\|_\infty \big\}$. 
It is clear that the norm on the right hand side can be bounded below by 
\begin{align*} 
\sup_{\stackrel{x \neq y}{x,y\in [-a,a]^2}}
\frac{ \Big| 
\big( |x_1 - tf(x_2)|^\alpha - |x_1 - tg(x_2)|^\alpha \big) 
- 
\big( |y_1 - tf(y_2)|^\alpha - |y_1 - tg(y_2)|^\alpha \big) 
\Big|}{ | x - y |^\alpha } 
\end{align*}
Evaluating this expression at $x_2 = y_2 = c$ with $-a < c < a$ we get a further estimate 
from below by 
\begin{equation*}
\sup_{\stackrel{x_1 \neq y_1}{x,y\in [-a,a]^2 }}
\frac{ 
|( |x_1 - tf(c)|^\alpha - |x_1 - tg(c)|^\alpha ) 
- 
( |y_1 - tf(c)|^\alpha - |y_1 - tg(c)|^\alpha )| 
}{ |x_1 - y_1|^\alpha } 
\end{equation*}
and evaluating once again at the points $x_1 = tg(c)$ and $y_1 = tf(c)$ 
we obtain a final lower bound 
\begin{equation*}
\geq 
\frac{ t^{\alpha} |g(c) - f(c)|^\alpha + t^\alpha |f(c) - g(c)|^\alpha }{ t^\alpha |f(c) - g(c)|^\alpha } 
= 2. 
\end{equation*}
Since this inequality holds for all $t$ in an open interval we conclude that the essential supremum of 
the norm 
$\|u(t)-v(t)\|_{1,\alpha}$ is bounded away from zero which proves Theorem \ref{LIP}. 
\end{proof} 
% 

%%%%%%%%%%%%%%%%%%
%%%%%%%%%%%%%%%%%%%%%%%%%%%
\section{Local well-posedness in $c^{1,\alpha}$} 
\label{Holder-wp} 

We turn to the question of well-posedness of \eqref{E} in the sense of Hadamard. 
As mentioned in the Introduction, local existence and uniqueness results in H\"older spaces 
are well known and our contribution here concerns only the continuity property of the solution map. 
To this end we will make some adjustments in the approach based on the particle-trajectory method of \cite{MB}. 
We first state Theorem \ref{LWP} more precisely as follows.   

\begin{theorem} \label{T4} 
Let $0<\alpha<1$. For any divergence free vector field $u_0 \in c^{1,\alpha}(\mathbb{R}^n)$ 
with compactly support vorticity there exist $T>0$ and a unique solution $u$ of \eqref{E} such that 
the map $u_0 \to u$ is continuous from $c^{1,\alpha}(\mathbb{R}^n)$ to $C([0,T), c^{1,\alpha}(\mathbb{R}^n))$. 
\end{theorem} 
\begin{proof}[Proof of Theorem \ref{T4}] 
We will concentrate on the two-dimensional case since the arguments in the three-dimensional case 
are very similar (the necessary modifications will be described below). 
We begin by constructing the Lagrangian flow as a unique solution of an ordinary differential equation 
in $\mathscr{U}_\delta$. 

Since in two dimensions the vorticity is conserved by the flow\footnote{That is, $\omega(t,\eta(t,x)) = \omega_0(x)$.} 
we can rewrite equations \eqref{IDE} in the form 
\begin{align} \label{IDE-2d} 
&\frac{d\eta}{dt}(t,x) 
= 
\int_{\mathbb{R}^2} K_2\big(\eta(t,x) - \eta(t,y) \big) \omega_0(y) \, dy 
=: F_{u_0}(\eta_t)(x), 
\\ \nonumber 
&\eta(0, x) = x 
\end{align} 
where $\omega_0=\nabla^\perp \cdot u_0$ and $K_2$ is given by \eqref{eq:KBS2}. 
In order to apply Picard's method of successive approximations it is sufficient to show that the right hand side of 
\eqref{IDE-2d} is locally Lipschitz continuous in $\mathscr{U}_\delta$. The first task is thus to establish that 
$F_{u_0}$ maps into $c^{1,\alpha}$. 

Changing variables in the integral we have 
\begin{align} \label{eq:F} 
F_{u_0}(\eta)(x) 
= 
(K_2\ast \tilde{\omega}_0) \circ \eta(x) 
\end{align} 
where $\tilde{\omega}_0 = \omega_0\circ\eta^{-1} \det{D\eta^{-1}}$ and $\eta \in \mathscr{U}_\delta$. 
Using \eqref{eq:KBS2} and \eqref{eq:jac} and estimating directly we obtain 
\begin{align} \label{eq:FF0} 
\| F_{u_0}(\eta)\|_\infty 
= 
\| K_2 \ast \tilde{\omega}_0 \|_\infty 
\leq 
C \|\omega_0\|_\infty 
\end{align} 
where $C$ depends on the size of the support of $\omega_0$. 
Next, differentiating $F_{u_0}$ in \eqref{eq:F} with respect to the $x$ variable gives 
\begin{align} \label{eq:DF} 
DF_{u_0}(\eta)(x) 
&= 
D(K_2\ast\tilde{\omega}_0)\circ\eta(x) D\eta(x) 
\\ \nonumber 
&= 
\Big( T\tilde{\omega}_0 - \frac{1}{2} \, \tilde{\omega}_0 J \Big) \circ\eta(x) D\eta(x) 
\end{align} 
where 
$\tiny J= \big( \begin{matrix} 0 & 1 \\ -1 & 0 \end{matrix} \big)$ 
is the standard $2\times 2$ symplectic matrix and $T$ is a singular integral operator of 
the Calderon-Zygmund type with the matrix kernel $DK_2(x) = \Omega(x)/|x|^2$ 
where $\Omega$ is homogeneous of degree zero 
$$ 
\quad 
Tf(x) = \frac{1}{2\pi}~ p.v. \int_{\mathbb{R}^2} \frac{\Omega(x-y)}{|x-y|^2} f(y) \, dy, 
\qquad 
\Omega(x) 
= 
|x|^{-2} \left(\begin{matrix} 2x_1x_2 & x_2^2 - x_1^2 \\ x_2^2 - x_1^2 & -2x_1x_2 \end{matrix} \right). 
$$ 
Standard estimates in H\"older spaces for such operators\footnote{See e.g., \cite{MB}, Chap. 4.} 
give 
\begin{align} \label{eq:DF1} 
&\| DF_{u_0}(\eta)\|_\infty 
\leq 
C\big( 
\|T\tilde{\omega}_0\circ\eta\|_\infty + \|\tilde{\omega}_0\circ\eta\|_\infty 
\big) \|D\eta\|_\infty 
\leq 
C\big( \|\tilde{\omega}_0\|_\infty + [\tilde{\omega}_0]_\alpha \big) 
\end{align} 
and 
\begin{align} \nonumber 
[DF_{u_0}(\eta)]_\alpha 
&\leq 
\| \big( T\tilde{\omega}_0 - 1/2\tilde{\omega}_0 J\big)\circ\eta\|_\infty [D\eta]_\alpha 
+ 
\big[ \big( T\tilde{\omega}_0 - 1/2\tilde{\omega}_0J\big)\circ\eta\big]_\alpha \|D\eta\|_\infty 
\\ \label{eq:DF2} 
&\leq 
C \big( \|\tilde{\omega}_0\|_\infty + [\tilde{\omega}_0]_\alpha \big) [D\varphi_\eta]_\alpha 
+ 
C \|D\eta\|_\infty^{1+\alpha} \big[T\tilde{\omega}_0 -1/2 \tilde{\omega}_0J \big]_\alpha 
\\ \nonumber 
&\leq 
C \big( \|\tilde{\omega}_0\|_\infty + [\tilde{\omega}_0]_\alpha \big) [D\varphi_\eta]_\alpha 
+ 
C(1+ \|D\varphi_\eta\|_\infty)^{1+\alpha} [\tilde{\omega}_0]_\alpha. 
\end{align} 
Furthermore, since from a direct computation using \eqref{eq:jac} we have 
\begin{align} \label{eq:to} 
[\tilde{\omega}_0]_\alpha 
&= 
[\omega_0\circ\eta^{-1} \det{D\eta^{-1}} ]_\alpha 
%\\ \nonumber 
%&\lesssim 
%\|\omega_0\circ\eta^{-1}\|_\infty [\det{D\eta^{-1}}]_\alpha 
%+ 
%[\omega_0\circ\eta^{-1}]_\alpha \|\det{D\eta^{-1}}\|_\infty 
%\\ \nonumber 
%&\lesssim 
\lesssim 
[D\varphi_\eta]_\alpha \| \omega_0\|_\infty + [\omega_0]_\alpha 
\end{align} 
combining these estimates with \eqref{eq:FF0} and the fact that $\eta \in \mathscr{U}_\delta$ 
we get 
\begin{align} \label{eq:FF00} 
\|F_{u_0}(\eta)\|_{1,\alpha} 
\lesssim 
C \big( \|\omega_0\|_\infty + [\omega_0]_\alpha \big). 
\end{align} 

To show that $F_{u_0}$ maps $\mathscr{U}_\delta$ into $c^{1,\alpha}(\mathbb{R}^2)$ 
it suffices now to observe that \eqref{eq:DF2} together with \eqref{eq:to} yield 
\begin{equation}\label{small holder property}
\lim_{h\to 0} \sup_{0<|x-y|<h} 
\frac{ |DF_{u_0}(\eta)(x) - DF_{u_0}(\eta)(y)|}{|x-y|^\alpha} = 0
\end{equation}
since both $\varphi_\eta$ and $\omega_0$ are in $c^{1,\alpha}(\mathbb{R}^2)$ by assumption. 

Finally, differentiating $F_{u_0}$ in \eqref{IDE-2d} with respect to $\eta$ in the direction 
$w \in c^{1,\alpha}(\mathbb{R}^2)$ we obtain 
\begin{align} \label{deltaF-2d} 
\delta_w F_{u_0}(\eta) (x) 
&= 
\frac{d}{dr} F_{u_0}(\eta + rw) (x) \Big\vert_{r=0} 
\\ \nonumber 
&= 
\int_{\mathbb{R}^2} DK_2\big( \eta(x) - \eta(y) \big) (w(x)-w(y))  \omega_0(y) dy 
\end{align} 
which again can be bounded directly using standard H\"older estimates by 
\begin{equation} \label{Gat} 
\| \delta_w F_{u_0}(\eta) \|_{1,\alpha} 
\lesssim 
C \big( \|\omega_0\|_\infty + [\omega_0]_\alpha \big) \| w \|_{1,\alpha}.
\end{equation} 
In particular, it follows that $F_{u_0}$ has a bounded Gateaux derivative in $\mathscr{U}_\delta$ 
and hence is locally Lipschitz by the mean value theorem. 

\smallskip 
We now turn to the question of dependence of the solutions of \eqref{E} on $u_0$. 
Note that since 
$\omega_0 = \nabla^\perp \cdot u_0$ 
the initial velocity $u_0$ appears as a parameter on the right hand side of \eqref{IDE-2d}. 
Moreover, since the dependence is linear it follows that continuity (and, in fact, differentiability) 
of the map $u_0 \to F_{u_0}$ is an immediate consequence of the estimate in \eqref{eq:FF00}. 
Applying the fundamental theorem of ordinary differential equations (with parameters) for Banach spaces 
we find that there exist $T>0$ and a unique Lagrangian flow $\eta \in C([0,T), \mathscr{U}_\delta)$ 
which depends continuously (in fact, differentiably) on $u_0$. 
Using the equations in \eqref{IDE-2d} we find that the same is true of the time derivative 
$\dot{\eta} \in C([0,T),c^{1,\alpha}(\mathbb{R}^2))$. 
It follows therefore that the vector field $u=\dot{\eta}\circ\eta^{-1}$ belongs to 
$C([0,T), c^{1,\alpha}(\mathbb{R}^2)) \cap C^1([0,T), c^\alpha(\mathbb{R}^2))$ 
and a routine calculation shows that it is divergence free. 

Next, suppose that $u_0$ and $v_0$ are two divergence free vector fields in $c^{1,\alpha}(\mathbb{R}^2)$ 
and let $\eta(t)$ and $\xi(t)$ be the corresponding Lagrangian flows solving the Cauchy problem \eqref{IDE-2d} 
in $\mathscr{U}_\delta$ with initial vorticities 
$\nabla^\perp\cdot u_0$ and $\nabla^\perp\cdot v_0$ respectively. 
Given any $\varepsilon >0$ and using the fact that smooth functions are dense in $c^{1,\alpha}$ 
we can choose $\phi_\varepsilon$ in $C^\infty([0,T)\times \mathbb{R}^2)$ such that 
$$
\sup_{0\leq t \leq T} \| \phi_\varepsilon(t) - \dot{\eta}(t)\|_{1,\alpha} < \varepsilon. 
$$ 
Applying this together with \eqref{comp1} and \eqref{comp3} we estimate 
\begin{align*} 
\| u - v \|_{1,\alpha} 
&= 
\| \dot{\eta}\circ\eta^{-1} - \dot{\xi}\circ\xi^{-1} \|_{1,\alpha} 
\\ 
&\leq 
\| \dot{\eta}\circ\eta^{-1} - \phi_\varepsilon\circ\eta^{-1} \|_{1,\alpha} 
+ 
\| \phi_\varepsilon\circ\eta^{-1} - \phi_\varepsilon\circ\xi^{-1} \|_{1,\alpha} 
+ 
\| \phi_\varepsilon\circ\xi^{-1} - \dot{\xi}\circ\xi^{-1} \|_{1,\alpha} 
\\ 
&\lesssim 
\| \dot{\eta} - \phi_\varepsilon \|_{1,\alpha} 
+ 
\| \phi_\varepsilon\|_{2,\alpha} \| \eta^{-1} - \xi^{-1} \|_{1,\alpha} 
+ 
\| \phi_\varepsilon - \dot{\xi} \|_{1,\alpha}.  
\end{align*} 
The first term of the last line is clearly bounded by $\varepsilon$. 
The middle term converges to zero by Lemma \ref{lemHC} (continuity of the inversion map) 
and the fact that $\eta$ converges to $\xi$ in $c^{1,\alpha}$ whenever $u_0$ converges to $v_0$ 
since the Lagrangian flows $\eta$ and $\xi$ depend continuously on the initial velocities. 
To dispose of the last term we use \eqref{IDE-2d} and the triangle inequality 
\begin{align*} 
\| \phi_\varepsilon - \dot{\xi} \|_{1,\alpha} 
&\leq 
\| \phi_\varepsilon - \dot{\eta} \|_{1,\alpha} + \| \dot{\eta} - \dot{\xi} \|_{1,\alpha}  
\\ 
&\leq 
\varepsilon 
+ 
\| F_{u_0}(\eta) - F_{v_0}(\eta) \|_{1,\alpha} 
+ 
\| F_{v_0}(\eta) - F_{v_0}(\xi) \|_{1,\alpha}. 
\end{align*} 
Applying \eqref{eq:FF00} we obtain 
\begin{align*} 
\| F_{u_0}(\eta) - F_{v_0}(\eta) \|_{1,\alpha} 
&= 
\big\| \int_{\mathbb{R}^2} 
K_2\big( \eta(t, \cdot) - \eta(t,y) \big) \big( \nabla^\perp\cdot u_0(y) - \nabla^\perp\cdot v_0(y) \big) 
dy \big\|_{1,\alpha} 
\\ 
&\lesssim 
\| \nabla^\perp\cdot (u_0 - v_0) \|_\infty 
+ 
 \big[ \nabla^\perp\cdot (u_0 - v_0) \big]_\alpha 
 \\ 
 &\lesssim 
 \| u_0 - v_0 \|_{1,\alpha} 
\end{align*} 
and using \eqref{Gat} we get 
\begin{align*} 
\| F_{v_0}(\eta) - F_{v_0}(\xi) \|_{1,\alpha} 
&= 
\big\| \int_0^1 \frac{d}{dr} F_{v_0}\big( r\eta + (1-r)\xi \big) dr \big\|_{1,\alpha} 
\\ 
&\leq 
\int_0^1 \| \delta_{\eta - \xi} F_{v_0} \big( r\eta + (1-r)\xi \big) \|_{1,\alpha} dr 
\\ 
&\lesssim 
\| v_0\|_{1,\alpha} \| \eta - \xi \|_{1,\alpha}, 
\end{align*} 
where the latter converges to zero by continuous dependence of the flows on $u_0$ and $v_0$ as before. 
This completes the proof of Theorem \ref{T4} when $n=2$. 

%%%%%%%%%%%%%%%%%%%%%
\smallskip 
In the three-dimensional case the flow equations \eqref{IDE} take only a slightly more complicated form 
\begin{align} \label{IDE-3d} 
&\frac{d\eta}{dt}(t,x) 
= 
\int_{\mathbb{R}^3} K_3\big(\eta(t,x) - \eta(t,y) \big) D\eta(t,y) \omega_0(y) \, dy =: G_{u_0}(\eta_t)(x) 
\\ \nonumber 
&\eta(0, x) = x 
\end{align} 
where $\omega_0=\nabla \times u_0$, $K_3$ is given by \eqref{eq:KBS3} 
and consequently the derivative $\delta_w G_{u_0}$ 
in the direction $w\in c^{1,\alpha}(\mathbb{R}^3)$ has an extra term 
\begin{align} \label{deltaF-3d} 
\delta_w G_{u_0}(\eta) (x) 
&= 
\int_{\mathbb{R}^3} DK_3\big( \eta(x) - \eta(y) \big) (w(x)-w(y))  D\eta(y) \omega_0(y) \, dy 
\\ \nonumber 
&+ 
\int_{\mathbb{R}^3} K_3\big( \eta(x) - \eta(y) \big) Dw(y) \omega(y) \, dy. 
\end{align} 
As before, applying standard H\"olderian estimates we obtain the analogues of \eqref{eq:FF00} and \eqref{Gat} 
and the proof proceeds as in the two dimensional case. 
\end{proof} 
\begin{remark} \label{noncomp} 
Theorem \ref{T4} remains valid if the initial vorticity has noncompact support and satisfies some suitable 
decay conditions at infinity. In Section \ref{Besov} we will apply it under the assumption 
$\omega_0 \in L^1(\mathbb{R}^2)$. 
In this case we only have to replace the bound in \eqref{eq:FF0} with 
\begin{equation} \label{eq:FF0'} \tag{4.3'} 
\| F_{u_0} (\eta) \|_\infty 
\lesssim 
\|\tilde{\omega}_0\|_\infty + \| \tilde{\omega}_0\|_{L^1} 
\end{equation} 
and those in \eqref{eq:DF1}, \eqref{eq:DF2} with 
\begin{align} \label{eq:DF1'} \tag{4.5'} 
\| DF_{u_0}(\eta)\|_\infty 
&\lesssim 
\|\tilde{\omega}_0\|_{0,\alpha} + \|\tilde{\omega}_0\|_{L^1} 
\\ \label{eq:DF2'} \tag{4.6'} 
[DF_{u_0}(\eta)]_\alpha 
&\lesssim 
\big( \|\tilde{\omega}_0\|_{0,\alpha} + \|\tilde{\omega}_0\|_{L^1} \big) [D\varphi_\eta]_\alpha 
+ 
\big( 1 + \|D\varphi_\eta\|_\infty \big)^{1+\alpha} [\tilde{\omega}_0]_\alpha 
\end{align} 
and adjust the estimates \eqref{eq:FF00} and \eqref{Gat} accordingly. 
The rest of the proof remains unchanged. 
\end{remark} 
% 
%\begin{remark}\label{commutator inequalities}
%We also mention in passing that an alternative proof of Theorem \ref{T4} could be given based on 
%the following commutator inequalities 
% 
%\begin{align*} 
%\| D^s(fg)-f D^sg \|_\infty 
%\lesssim 
%\|f\|_{\dot{B}^1_{\infty,\infty}}\|g\|_{\dot{B}^{s-1}_{\infty,\infty}} 
%+ 
%\|f\|_{\dot{B}^s_{\infty,\infty}}\|g\|_{\dot{B}^0_{\infty,\infty}} 
%\end{align*} 
% 
%and 
% 
%\begin{align*} 
%\|D^s (v\cdot\nabla g)-(v\cdot \nabla)D^s g\|_\infty 
%\lesssim 
%\|v\|_{C^{1,s}} \|g\|_{C^{1,s}}.
%\end{align*} 
% 
%where $D^s = (-\Delta)^{s/2}$ is the fractional Laplacian with $0<s<1$. 
%Inequalities of this type were recently obtained in \cite{BL-commutator}. 
%However, using this approach we could only establish local well-posedness of \eqref{E} 
%in the space $c^{2,\alpha}$ for $0<\alpha<1$, under suitable decay conditions. 
%Besides continuity of the double Riesz transforms and the density property 
%our proof required that the vorticity of $u$ be at least of class $C^{1,\alpha}$ and hence 
%that $u$ belong to $c^{2,\alpha}$ (in order to apply the second of the commutator estimates above). 
%A similar proof also works in $W^{s,p}$ for $p \geq 2$ and $s>2/p+2$. 
%\end{remark} 
% 

%%%%%%%%%%%%%%%%%%
%%%%%%%%%%%%%%%%%%%%%%%%%%%
\section{Proof of Theorem \ref{LIP-Besov}: weak norm inflation in $B^2_{2,1}$} 
\label{Besov} 

Our aim in this section is to exhibit a norm inflation type mechanism which involves a family of Besov spaces. 
On the one hand the result stated in Theorem \ref{LIP-Besov} is weaker than the 
instantaneous blowup results 
obtained by Bourgain and Li. On the other hand, our method is applicable in the borderline function spaces 
that were left out of the analysis in \cite{BL, BL1}. 
The proof involves constructing a Lagrangian flow with a large deformation gradient and 
a high-frequency perturbation of the corresponding initial vorticity. It also relies on the continuity result 
for the solution map in the little H\"older space of Section \ref{Holder-wp}. 

It will be convenient to work with the vorticity equations which in two dimensions have the form 
\begin{align} \label{B2:eq:euler-v} 
&\omega_t + u{\cdot}\nabla \omega = 0, 
\qquad\qquad\qquad\quad 
t \geq 0, \; x \in \mathbb{R}^2 
\\  \nonumber
&\omega(0) = \omega_0 
\end{align} 
where $u = \nabla^\perp \Delta^{-1} \omega$ and $\omega = \partial_1 u_2 - \partial_2 u_1$. 
We first proceed to choose the initial vorticity $\omega_0$ for the Cauchy problem \eqref{B2:eq:euler-v}. 
Given any smooth radial bump function $0 \leq \phi \leq 1$ with support in the ball $B(0,1/4)$ 
let 
\begin{align} \label{B2:eq:bump} 
\phi_0(x_1, x_2) 
= 
\sum_{\varepsilon_1, \varepsilon_2 = \pm 1} 
\varepsilon_1 \varepsilon_2 \phi(x_1 {-} \varepsilon_1, x_2 {-} \varepsilon_2). 
\end{align} 
Clearly, the function $\phi_0$ is odd with respect to both $x_1$ and $x_2$. 
Given any $M \gg 1$, $r>0$ and $q>0$ define 
\begin{align} \label{B2:eq:iv} 
\omega_0(x) 
= 
\omega^{M,N,r,q}_0(x) 
= 
M^{-2} N^{-\frac{1}{q}} \sum_{0 \leq k \leq N} \phi_k(x) 
\end{align} 
where $N = 1, 2 \dots$ and 
$\phi_k(x) = 2^{( - {1} + \frac{2}{r})k} \phi_0 (2^k x)$. 
Note that the supports of $\phi_k$ are disjoint and compact with 
\begin{equation} \label{B2:eq:suppp} 
\mathrm{supp}\,{ \phi_k } 
\subset 
\bigcup_{\varepsilon_1, \varepsilon_2 = \pm 1} 
B\big( (\varepsilon_1 2^{-k}, \varepsilon_2 2^{-k}), 2^{-(k+2)} \big). 
\end{equation} 

Next, we have 
\begin{lem} \label{B2:lem:omega-0} 
If $1<q<\infty$ and $2<r<\infty$ with $q \leq r$. For any integer $N>0$ we have  
\begin{align} \label{B2:eq:omega-0} 
\|\omega_0\|_{W^{1,r}} + \| \omega_0 \|_{B^{1}_{r,q}} \lesssim  M^{-2}. 
\end{align} 
\end{lem} 
\begin{proof} 
We proceed to estimate the two terms on the left hand side of \eqref{B2:eq:omega-0} separately. 
Observe that the supports of $\phi_k$ in \eqref{B2:eq:suppp} are disjoint and therefore 
changing variables we get 
$$ 
\| \omega_0\|_{L^r}^r 
\simeq 
M^{-2r} N^{-\frac{r}{q}} \sum_{k=0}^N 2^{-kr} \int_{\mathbb{R}^2} |\phi_0(x)|^r dx 
\lesssim 
M^{-2r} 
$$ 
and since $q \leq r$ we similarly have 
$$ 
\| \partial_l \omega_0\|_{L^r}^r 
\simeq 
M^{-2r} N^{-\frac{r}{q}} \sum_{k=0}^N 2^{-kr} \int_{\mathbb{R}^2} |2^k \partial_l\phi_0(x)|^r dx 
\lesssim 
M^{-2r} 
$$ 
for $l = 1, 2$, which gives the required bound for the $W^{1,r}$ term. 

The estimate of the $B^1_{r,q}$ term is slightly more cumbersome. 
It will be convenient to work with the Fourier transform of $\omega_0$. 
In this case the supports of $\hat \phi_k$ are no longer disjoint, nevertheless each $\hat \phi_k$ 
can be decomposed into a "bump" part and a decaying "tail" part where the bump parts have disjoint supports. 
From \eqref{Besov-inh} and the calculations above we only need to estimate the homogeneous Besov norm 
$\|\omega_0\|_{\dot B^1_{r,q}}$. 
We have 
\begin{equation*} 
\hat\phi_k(\xi)=2^{(-3+\frac{2}{r})k}\hat\phi_0(2^{-k}\xi) 
\end{equation*} 
and, since $\hat\phi_0$ is a function of rapid decrease, given any $\alpha>0$ we can find $K_1>1$ 
such that 
\begin{equation*} 
|\hat \phi_0(\xi)|\leq C|\xi|^{-\alpha}\quad\text{for} 
\quad 
|\xi| \geq K_1. 
\end{equation*} 
Let $\alpha>3-2/r$. In this case we have 
\begin{equation*} 
|\hat\phi_k(\xi)|\leq  2^{(-3+\frac{2}{r})k} |\hat\phi_0(2^{-k}\xi)| 
\leq 
2^{(-3+\frac{2}{r}+\alpha)k}|\xi|^{-\alpha} 
\quad \text{for} \quad |\xi| \geq 2^{k}K_1. 
\end{equation*} 
Using the Hausdorff-Young inequality we get 
\begin{eqnarray*} 
\|\omega_0\|_{\dot B^1_{r,q}}^{q} 
&\leq& 
\sum_\ell  2^{\ell q} \| \psi_\ell \, \hat{\omega}_{0} \|_{L^{r'}}^{q} 
\\ 
&\lesssim& 
M^{-2q}N^{-1}\sum_\ell 
\Big\| \psi_\ell(\xi) \, |\xi| \sum_{k= 0}^N \hat\phi_k(\xi) \Big\|_{L^{r'}_\xi}^{q}
\\ 
&\simeq& 
M^{-2q}N^{-1}\sum_\ell \left( 
\int_{\mathbb{R}^2} |\psi_\ell(\xi)|^{r'} |\xi|^{r'} \Big| \sum_{k=0}^N \hat\phi_k(\xi) \Big|^{r'} d\xi 
\right)^{q/r'} 
\end{eqnarray*} 
where $1/r + 1/r' =1$. 
Given any integers $k_1$, $k_2$ and $K$ (with $k_2 \leq K$) introduce the functions 
$$ 
\Phi_{k_1,k_2}^K(\xi) 
= 
\chi_{[2^{k_1},2^{k_1+1}]}(\xi) \, |\xi| \sum_{k=k_2}^K |\hat\phi_k (\xi)|. 
$$ 
A direct calculation yields 
\begin{eqnarray*}
\Phi^0_{j-K,-\infty}(2^{-K}\xi) 
&=& 
\chi_{[2^{j-K},2^{j-K+1}]}(2^{-K}\xi) \, \big| 2^{-K}\xi \big| \sum_{k=-\infty}^0 \big| \hat{\phi}_k (2^{-K}\xi) \big| 
\\ 
&=& 
2^{-K} \chi_{[2^j,2^{j+1}]}(\xi) \, |\xi| \sum_{k=-\infty}^0 2^{(-3+\frac{2}{r})k} \big| \hat{\phi}_0(2^{-(k+K)}\xi) \big| 
\\ 
&=&  
2^{-K} 2^{(3-\frac{2}{r})K} 
\chi_{[2^j,2^{j+1}]}(\xi) \, |\xi| \sum_{k=-\infty}^K 2^{(-3+\frac{2}{r})k} |\hat{\phi}_0(2^{-k}\xi)| 
\\ 
&=& 
2^{\frac{2K}{r'}} \Phi^K_{j,-\infty}(\xi) 
\end{eqnarray*}
% 
%with $k'=k+K$. 
which leads to the following scaling identity 
\begin{equation*} 
\Phi_{j,-\infty}^K(\xi) 
= 
2^{-\frac{2K}{r'}} \Phi_{j-K,-\infty}^0(2^{-K}\xi) 
\end{equation*} 
for $j>K$ and similarly we have 
\begin{equation*} 
\Phi_{j ,-\infty}^\infty(\xi) 
= 
2^{-\frac{2j}{r'}} \Phi_{0,-\infty}^\infty(2^{-j} \xi). 
\end{equation*} 
The above will be needed below in order to control the tail parts (for both high and low frequencies).  
\begin{theorem*} 
\begin{equation*} 
\sum_{j\geq 1}\|\Phi_{j,-\infty}^0\|_{L^{r'}}^q < \infty, 
\quad 
\sum_{j<1} \|\Phi_{j,0}^\infty\|_{L^{r'}}^q < \infty 
\quad \text{and} \quad 
\|\Phi_{0,-\infty}^\infty\|_{L^{r'}}^q \lesssim 1. 
\end{equation*} 
\end{theorem*} 
\begin{proof}[Proof of Claim] 
We have 
\begin{eqnarray*} 
|\Phi_{0,-\infty}^\infty(\xi)| 
&=& 
\chi_{[1,2]}(\xi) \, |\xi|\sum_{-\infty\leq k \leq \infty} |\hat\phi_k(\xi)| 
\\ 
&\lesssim& 
\chi_{[1,2]}(\xi) 
\Bigg(\sum_{-\infty\leq k \leq -\log_2 K_1} + \sum_{-\log_2 K_1\leq k \leq 1} 
+
\sum_{1\leq k \leq \infty} \Bigg) |\hat\phi_k(\xi)| 
\\ 
&\lesssim& 
\chi_{[1,2]}(\xi) 
\Bigg(\sum_{-\infty\leq k \leq -\log_2 K_1} 2^{(-3+\frac{2}{r}+\alpha)k} |\xi|^{-\alpha} 
+ \text{finite sum} + 
\sum_{1\leq k \leq \infty} 2^{(-3+\frac{2}{r})k} \Bigg) 
\\ 
&\lesssim& 
\chi_{[1,2]}(\xi) 
\end{eqnarray*} 
so that $\|\Phi_{0,-\infty}^\infty\|_{L^{r'}}^q \lesssim 1$. 

Next, since $2^kK_1<K_1$ for $k\leq 0$, we have 
\begin{eqnarray*} 
|\Phi_{j,-\infty}^0(\xi)|\lesssim 
\chi_{[2^j, 2^{j+1}]}(\xi) \, |\xi|\sum_{k=-\infty}^02^{(-3+\frac{2}{r}+\alpha)k} |\xi|^{-\alpha} 
\lesssim 
\chi_{[2^j, 2^{j+1}]}(\xi) \, |\xi|^{-\alpha+1} 
\end{eqnarray*} 
for $|\xi|>K_1$ and using this estimate we get $\sum_{j>1}\|\Phi_{j,-\infty}^0\|_{L^{r'}}^q < \infty$. 

Finally, we have 
\begin{equation*} 
|\Phi_{j,0}^\infty(\xi)| 
\lesssim 
\chi_{[2^j, 2^{j+1}]}(\xi) \, |\xi| \sum_{k=0}^\infty2^{(-3+\frac{2}{r})k} 
\lesssim 
\chi_{[2^j, 2^{j+1}]}(\xi) \, |\xi| 
\end{equation*} 
and so we obtain $\sum_{j<1}\|\Phi_{j,0}^\infty\|_{L^{r'}}^q < \infty$ as before. 
\end{proof} 

We now return to the proof of Lemma \ref{B2:lem:omega-0}. Using the fact that the supports are disjoint 
we have 
\begin{align*} 
\bigg| \sum_{k=0}^N & |\xi|\hat\phi_k(\xi)\bigg|^{r'} 
\leq 
\sum_{j}\bigg| \sum_{k=0}^N \chi_{[2^{j},2^{j+1}]}(\xi) \, |\xi|\hat\phi_k(\xi) \bigg|^{r'} 
\\ 
&\leq 
|\Phi_{1,0}^N(\xi)|^{r'}+|\Phi_{2,0}^N(\xi)|^{r'} + \cdots + |\Phi_{N,0}^N(\xi)|^{r'} 
+ 
\sum_{j>N}|\Phi_{j,0}^N(\xi)|^{r'} + \sum_{j<1}|\Phi_{j,0}^N(\xi)|^{r'} 
\\ 
&= I_1(\xi) + I_2(\xi) + I_3(\xi) 
\end{align*} 
and consequently (note that $\psi_\ell$ are also essentially disjoint) 
\begin{align*} 
\|\omega_0\|_{\dot{B}^1_{r,q}}^q 
&\lesssim 
M^{-2q} N^{-1} \sum_\ell \bigg( \int_{\mathbb{R}^2} |\psi_\ell(\xi)|^{r'} I_1(\xi) d\xi \bigg)^{q/r'} 
\\ 
&+ 
M^{-2q} N^{-1} \sum_\ell \bigg( 
\int_{\mathbb{R}^2} |\psi_\ell(\xi)|^{r'} \big( I_2(\xi) + I_3(\xi) \big) d\xi 
\bigg)^{q/r'}. 
\end{align*} 
Note that $I_1$ is a finite sum of bump parts while $I_2$ and $I_3$ correspond to the two decaying tails. 
Using the Claim and the scaling identity together with the formula for $\Phi^K_{k_1,k_2}$ we can estimate 
the first of the integrals on the right hand side of the expression above by 
\begin{align*} 
\sum_{\ell = 1}^N 
\Bigg( \int_{\mathbb{R}^2} &|\psi_\ell(\xi)|^{r'} I_1(\xi) d\xi\Bigg)^{q/r'} 
\\ 
&\leq 
\sum_{\ell = 1}^N \Bigg( 
\int|\psi_\ell(\xi)|^{r'}\bigg(|\Phi_{1,-\infty}^\infty(\xi)|^{r'} 
+ 
|\Phi_{2,-\infty}^\infty(\xi)|^{r'} + \cdots + |\Phi_{N,-\infty}^\infty(\xi)|^{r'} \bigg) 
d\xi \Bigg)^{q/r'} 
\\ 
&\lesssim 
\sum_{\ell=1}^N \bigg( 
\int|\Phi_{\ell,-\infty}^\infty(\xi)|^{r'} 
d\xi\bigg)^{q/r'} 
\lesssim 
N 
\end{align*} 
and similarly 
\begin{align*} 
\sum_{\ell<1,N<\ell} \Bigg( 
\int_{\mathbb{R}^2} &|\psi_\ell(\xi)|^{r'} \big( I_2(\xi) + I_3(\xi) \big) d\xi 
\Bigg)^{q/r'} 
\\ 
&\leq 
\sum_{\ell<1,N<\ell} \Bigg( 
\int|\psi_\ell(\xi)|^{r'} \bigg( 
\sum_{j>N}|\Phi_{j,-\infty}^N(\xi)|^{r'} + \sum_{j<1}|\Phi_{j,0}^\infty(\xi)|^{r'} \bigg) d\xi 
\Bigg)^{q/r'} 
\\ 
&\lesssim 
\sum_{\ell<1} \bigg( 
\int 
|\Phi_{\ell,0}^\infty(\xi)|^{r'} d\xi 
\bigg)^{q/r'} 
+
\sum_{N<\ell} \bigg( 
\int 
|\Phi_{\ell,-\infty}^N(\xi)|^{r'} d\xi 
\bigg)^{q/r'} 
\\ 
&\lesssim 
C, 
\end{align*} 
where $C>0$ is independent of $N$. 
Combining the above estimates we get 
\begin{equation*} 
\| \omega_0 \|_{\dot B^1_{r,q}} \lesssim M^{-2} 
\end{equation*} 
which together with the $L^r$ bound of $\omega_0$ gives the desired bound. 
\end{proof} 

In particular, since $r>2$ it follows from Lemma \ref{B2:lem:omega-0} that the associated velocity field 
$u = \nabla^\perp\Delta^{-1} \omega \in W^{2,r}$ 
has a $C^1$ smooth Lagrangian flow $\eta(t)$ obtained by solving the flow equations 
\begin{equation} \label{FLx} 
\frac{d\eta}{dt}(t,x) = u(t,\eta(t,x)), 
\qquad \eta(0,x)=x. 
\end{equation} 
Furthermore, it is not difficult to verify that $\eta(t)$ is hyperbolic with a stagnation point at the origin 
and preserves both $x_1$ and $x_2$ axes as well as the odd symmetries of $\omega_0$. 
\begin{prop} \label{B2:prop:Lag} 
Given $M \gg 1$ and $1 < q < \infty$ we have 
$$ 
\sup_{0 \leq t \leq M^{-3}} \| D\eta(t) \|_\infty > M 
$$ 
for any sufficiently large integer $N>0$ in \eqref{B2:eq:iv} and any $2<r<\infty$ sufficiently close to $2$. 
\end{prop} 
\begin{proof} 
The proof is a repetition (with obvious adjustments) of that given in \cite{MY2}; Prop. 6, 
for the special case $q=r>2$ 
;
see also \cite[Lemma 3.2]{BL} for the estimate of $R_{ii}\omega$.
It will be omitted. 
\end{proof} 
We will also need the following simple consequence of Gronwall's inequality 
(see \cite{BL}; Lemma 4.1 for example) 
\begin{lem} \label{Gron} 
If $u$ and $\tilde{u}$ are smooth divergence free vector fields on $\mathbb{R}^2$ and 
$\eta(t)$ and $\tilde{\eta}(t)$ are the corresponding solutions of \eqref{FLx} then 
$$ 
\sup_{0\leq t \leq 1}\| \eta(t) - \tilde{\eta}(t)\|_{C^1} 
\leq 
C\sup_{0\leq t \leq 1} \|u(t) - \tilde{u}(t)\|_{C^1} 
$$ 
where $C>0$ depends only on the $L^\infty$ norm of $u$ and $\tilde u$ and its derivatives. $\qquad\square$ 
\end{lem} 
%%%%%%%
\begin{proof}[Proof of Theorem \ref{LIP-Besov}] 
Let $M_j \nearrow \infty$. Choose any $N \gg 1$ and any sequences $r_j \searrow 2$ and $q_j \searrow 1$ 
such that the estimate of Proposition \ref{B2:prop:Lag} holds for the flow $\eta_j(t)$ 
of $u_j = \nabla^\perp\Delta^{-1}\omega_j$ 
where $\omega_j$ solves the vorticity equation \eqref{B2:eq:euler-v} with initial condition 
$\omega_{0,j} = \omega_0^{M_j,N,r_j,q_j}$ given by \eqref{B2:eq:iv}. 
For each $j \geq 1$ we will introduce a high-frequency perturbation of $\omega_{0,j}^n$ 
such that for any sufficiently large $n$ we have 
\begin{equation} \label{vort} 
\|\omega^n_{0,j}\|_{B^1_{r_j,q_j}} \lesssim 1 
\quad \text{and} \quad 
\|\omega^n_j(t^*)\|_{B^1_{r_j,q_j}} \gtrsim M_j^{1/3} 
\quad \text{for some $0<t^*\leq M_j^{-3}$}. 
\end{equation} 

To this end observe that we may assume 
\begin{equation} \label{B2:eq:assump} 
\|\omega_j(t) \|_{B^{1}_{r_j,q_j}} \leq M_j^{1/3} 
\qquad 
\text{for all} \;\; 
0 \leq t \leq M_j^{-3} 
\end{equation} 
or else there is nothing to prove. 
Using Proposition \ref{B2:prop:Lag} we can pick $0 \leq t^* \leq M_j^{-3}$ 
and a point $x^\ast = (x^\ast_1, x^\ast_2)$ 
for which the absolute value of one of the entries in $D\eta_j(t_0,x^\ast)$ is at least as large as $M_j$ 
and by continuity (since $r_j >2$) deduce that in a sufficiently small $\delta$-neighbourhood of 
$x^\ast$ we have 
\begin{align} \label{B2:eq:M} 
\left| \frac{\partial \eta_j^2}{\partial x_2} (t_0,x) \right| \geq M_j 
\qquad 
\text{for all} 
\quad 
|x-x^\ast| < \delta. 
\end{align} 
To construct a sequence of perturbations of $\omega_{0,j}$ in $B^{1}_{r_j,q_j}$ pick a smooth function 
$\hat{\chi} \in C^\infty_c(\mathbb{R}^2)$ with support in the unit ball such that 
$0 \leq \hat{\chi} \leq 1$ 
and 
$\int_{\mathbb{R}^2} \hat\chi(\xi) \, d\xi = 1$ 
and set 
\begin{equation} \label{eq:bp} 
\hat{\rho}(\xi) = \hat\chi(\xi - \xi_0) + \hat\chi(\xi + \xi_0), 
\qquad 
\text{where} \;\; 
\xi \in \mathbb{R}^2 
\;\; \text{and} \;\; 
\xi_0 = (2,0). 
\end{equation} 
Observe that $\mathrm{sup}\; \hat\rho \subset B(-\xi_0,1) \cup B(\xi_0,1)$ and 
\begin{equation} \label{eq:ro2} 
\rho(0) = \int_{\mathbb{R}^2} \hat{\rho}(\xi) \, d\xi = 2. 
\end{equation} 
For any $k \in \mathbb{Z}_+$ and $\lambda >0$ define 
\begin{equation} \label{eq:beta-pert} 
\beta_j^{k,\lambda}(x) 
= 
\frac{\lambda^{-1 + \frac{2}{r_j}}}{\sqrt{k}} 
\sum_{\varepsilon_1, \varepsilon_2 = \pm 1} 
\varepsilon_1 \varepsilon_2  \rho(\lambda(x-x^\ast_\epsilon)) \sin{kx_1} 
\end{equation} 
where 
$x^\ast_\epsilon = (\varepsilon_1 x^\ast_1, \varepsilon_2 x^\ast_2)$. 
\begin{lem} \label{B2:lem:rem}
Let  $1<q_j<2<r_j<\infty$, $2 \leq p \leq \infty$ and $\sigma >0$. 
For any sufficiently large $k \in \mathbb{Z}^+$ and $\lambda >0$ we have 
\begin{enumerate} 
\item[1.] 
$
\|\beta_j^{k,\lambda}\|_{W^{1,r_j}} 
\lesssim 
\|\beta_j^{k,\lambda}\|_{B^{1}_{r_j,q_j}} 
\lesssim 
k^{\frac{1}{2}} \lambda^{-1} 
$ 
\vskip 0.05cm 
\item[2.] 
$
\|\Delta^{\frac{1 + \sigma}{2}}\partial_l\Delta^{-1}\beta_j^{k,\lambda}\|_{L^p}
\lesssim 
k^{-\frac{1}{2}} \lambda^{-1+ \frac{2}{r_j} - \frac{2}{p}} (\lambda^\sigma + k^\sigma) 
$
\vskip 0.05cm 
\item[3.] 
$ 
\|\partial_l \Delta^{-1} \beta_j^{k,\lambda}\|_{L^p} 
\lesssim 
k^{-\frac{1}{2}} \lambda^{-2 + \frac{2}{r_j} - \frac{2}{p}} 
$ 
\end{enumerate} 
where $l = 1, 2$. 
\end{lem} 
\begin{proof}[Proof of Lemma \ref{B2:lem:rem}] 
The proof of the first inequality is similar to the second inequality which in turn is similar to that of Lemma \ref{B2:lem:omega-0}.
% that of Lemma \ref{B2:lem:omega-0}. 
To prove the second and third estimates it will be convenient to use the Fourier transform 
\begin{align} \label{B2:eq:FTb} 
\hat{\beta}_j^{k,\lambda}(\xi) 
= 
\frac{1}{2i} k^{-\frac{1}{2}} \lambda^{-3+\frac{2}{r_j}} 
\sum_{\varepsilon_1, \varepsilon_2 = \pm 1} \sum_{m=1}^2 
(-1)^{j+1} 
\varepsilon_1 \varepsilon_2 
\hat{\rho} \big( \lambda^{-1} \xi_m^k \big) 
e^{-2\pi i \langle x_\varepsilon^\ast, \xi_m^k \rangle} 
\end{align} 
where $\xi_m^k = \big( \xi_1 + \frac{(-1)^m}{2\pi}k, \xi_2 \big)$. 
Applying the Hausdorff-Young inequality we obtain 
\begin{align*} 
\big\| \Delta^{\frac{1 + \sigma}{2}}\partial_l \Delta^{-1} \beta_j^{k,\lambda} \big\|_{L^p} 
&\lesssim 
%\big\| |\cdot|^\sigma \hat\beta_j^{k,\lambda} \big\|_{L^{p'}} 
%\\ 
%&\lesssim 
k^{-\frac{1}{2}} \lambda^{-1+ \frac{2}{r_j}}\sum_{m=1}^2 \left(\int_{\mathbb{R}^2}\lambda^{-2p'} 
|\xi|^{\sigma p'} \big| \hat{\rho}(\lambda^{-1}\xi_m^k) \big|^{p'} \, d\xi\right)^{1/p'} 
\end{align*} 
where $1/p + 1/p' =1$. Changing the variables we further estimate by 
\begin{align*} 
&\lesssim 
k^{-\frac{1}{2}} \lambda^{-1+\frac{2}{r_j}-2(1-\frac{1}{p'})} 
\sum_{m=1}^2 
\left( \int_{\mathbb{R}^2} 
\bigg( 
\Big( \xi_1-\frac{(-1)^m}{2\pi}k \Big)^2
+ 
\xi_2^2 \bigg)^{\frac{\sigma p'}{2}} 
\big| \hat{\rho}(\lambda^{-1}\xi) \big|^{p'} 
\frac{d\xi}{\lambda^2} \right)^{1/p'} 
\\ 
&\lesssim 
k^{-\frac{1}{2}} \lambda^{-1+\frac{2}{r_j}-\frac{2}{p}} 
\sum_{m=1}^2 
\left( \int_{\mathbb{R}^2} 
\bigg( 
\Big(\lambda\xi_1-\frac{(-1)^m}{2\pi}k \Big)^2 
+ 
(\lambda\xi_2)^2 \bigg)^{\frac{\sigma p'}{2}} \big| \hat{\rho}(\xi) \big|^{p'} 
d\xi \right)^{1/p'} 
\\ 
&\lesssim 
k^{-\frac{1}{2}} \lambda^{-1+ \frac{2}{r_j} - \frac{2}{p}} 
\big( \lambda^\sigma + k^\sigma \big). 
\end{align*} 
Similarly, we also obtain 
\begin{align*} 
\big\| \partial_l &\Delta^{-1} \beta_j^{k,\lambda} \big\|_{L^p} 
\lesssim
%\big\| |\cdot|^{-1}\hat\beta_j^{k,\lambda} \big\|_{L^{p'}} 
%\\
%&\lesssim 
%k^{-\frac{1}{2}} \lambda^{-1+\frac{2}{r_j}-\frac{2}{p}} 
%\sum_{m=1}^2 \left(\int_{\mathbb{R}^2} 
%\left( \Big(\lambda\xi_1-\frac{(-1)^m}{2\pi}k \Big)^2 
%+ 
%(\lambda\xi_2)^2\right)^{-\frac{p'}{2}} \big| \hat{\rho}(\xi) \big|^{p'} 
%d\xi \right)^{1/p'} 
%\\ 
%&\lesssim 
k^{-\frac{1}{2}} \lambda^{-2+2(\frac{1}{r_j}-\frac{1}{p})}
\end{align*} 
for any sufficiently large $k$ and $\lambda$. 
\end{proof} 

Next, set $\beta_j^n = \beta_j^{k,\lambda}$ where $k=\lambda^2$, $\lambda=3n$ and $n\gg 1$. 
Using \eqref{B2:eq:M} and \eqref{eq:ro2} we now have 
\begin{lem} \label{LLem} 
Let $M_j$, $N$, $r_j$, $q_j$, $n$ and $t^*$ be as above. Then 
\begin{enumerate} 
\item[1.] 
$ 
\| \partial_2 \beta_j^n \partial_1\eta_j^2(t^*)\|_{L^{r_j}} \lesssim Cn^{-1} 
$ 
\vskip 0.05cm 
\item[2.] 
$ 
\|\partial_1\beta_j^n \partial_2\eta_j^2(t^*)\|_{L^{r_j}} 
\gtrsim 
M_j \big(1+\mathcal{O}(n^{-\frac{1}{2}})\big) - Cn^{-1} 
$ 
\end{enumerate} 
where $C$ depends on $\|\hat{\rho}\|_{L^{r'_j}}$ and 
$\sup_{0\leq t \leq 1} \|u_j(t)\|_{C^1}$ and $1/r'_j + 1/r_j =1$. 
\end{lem} 
\begin{proof} 
The proof is analogous to that in \cite{MY2}; Lem. 11. 
\end{proof} 

For each $j\geq 1$ define a perturbation sequence of initial vorticities 
$$ 
\omega_{0,j}^n (x) = \omega_{0,j}(x) + \beta_j^n(x), 
\qquad 
n \gg 1. 
$$ 
By Lemma \ref{B2:lem:omega-0} and Lemma \ref{B2:lem:rem} (part 1) it is in $B^1_{r_j,q_j}$ 
which shows the first of the inequalities in \eqref{vort}. 
Let $\omega_j^n \in C([0,1],B^1_{r_j,q_j}(\mathbb{R}^2))$ be the solution of the vorticity equations 
with initial data $\omega_{0,j}^n$. 
Recall that $r_j>2$ and $q_j>1$ are already fixed. Given any $p\geq 2$ 
pick $0<\sigma<1+1/p-1/r_j$ in Lemma \ref{B2:lem:rem} (parts 2 and 3) so that 
$\| \nabla^\perp\Delta^{-1}(\omega_{0,j}^n - \omega_{0,j})\|_{W^{1+\sigma,p}} \to 0$ 
as $n\to \infty$.\footnote{More precisely, observe that the power of $n$ 
(recall $k = \lambda^2 \simeq n$) on the right hand side of the inequality in part 2 of Lemma \ref{B2:lem:rem} is
\begin{equation*}
2(-1 + 1/r_j - 1/p + \sigma)
% < 
%2(-1/2 + 1/r_j - 1/p)
% <
 %2(-1/2 + 1/r_j) 
< 0 
\end{equation*}
so that the $L^p$-norm there goes to zero with $n \to \infty$.
Furthermore, 
% since
%$2/p<\sigma$, 
if we choose $p$ to satisfy $p > 2/\sigma$ (which is possible whenever $p>\frac{r_j}{r_j-1}$) then 
we have the embeddings 
$W^{1+\sigma,p}(\mathbb{R}^2) \subset C^{1,\sigma-\frac{2}{p}}(\mathbb{R}^2) \subset c^{1,\alpha}(\mathbb{R}^2)$ for any $0<\alpha<\sigma-2/p$.} 
Therefore, using continuity of the solution map in the little H\"older spaces of Theorem \ref{LWP}  
we find\footnote{Note that by construction $\beta_j^n \in \mathscr{S}(\mathbb{R}^2)$ has noncompact support, 
cf. Remark \ref{noncomp}.} 
\begin{equation} \label{eq:A} 
\sup_{0 \leq t \leq 1}\| \nabla^\perp\Delta^{-1} ( \omega_j^n(t) - \omega_j(t) ) \|_{C^1} 
\lesssim 
\sup_{0 \leq t \leq 1}\| \nabla^\perp\Delta^{-1} ( \omega_j^n(t) - \omega_j(t) ) \|_{1,\alpha} 
\longrightarrow 0 
\end{equation} 
as $n \to \infty$ and from Lemma \ref{Gron} we get 
\begin{align} \label{eq:LL} 
\theta_n = \sup_{0\leq t \leq 1} \| \eta_j^n(t) - \eta_j(t) \|_{C^1} 
\longrightarrow 0 
\quad 
\text{as} 
\; 
n \to \infty 
\end{align} 
where $\eta_j^n(t)$ is the flow of the velocity field $\nabla^\perp\Delta^{-1} \omega_j^n$. 
\begin{remark}
By following a well known argument of Kato and Ponce \cite{KP-Duke} we can show 
the incompressible Euler equations are locally well-posed
in the sense of Hadamard in $W^{s,p}(\mathbb{R}^2)$ with $p>2$ and $s>2+2/p$ or $p=2$ and $s>2$.
For details, see the Appendix.
We can also apply continuity of the solution map in $H^{1+\sigma}(\mathbb{R}^2)$ 
to \eqref{eq:A} directly.
However, continuity in $H^{1+\sigma}$ is currently known to be valid only in two dimensions 
due to appearance of the vortex-stretching term in the 3D case.
It therefore seems that the little H\"older spaces are the most suitable function space 
to study well-posedness in the sense of Hadamard so far. 
\end{remark}
Using \eqref{eq:LL}, conservation of vorticity and the fact that the flows are volume-preserving 
we have 
\begin{align} \nonumber  
\| \omega_j^n(t^*) \|_{B^1_{r_j,q_j}} 
&\gtrsim 
\| \nabla\omega_{0,j}^n \cdot \nabla^\perp \eta_j^{n,2}(t^*) \|_{L^{r_j}} 
\gtrsim 
\| \nabla\omega_{0,j}^n \cdot \nabla^\perp\eta_j^2(t^*) \|_{L^{r_j}} 
- 
\theta_n \|\nabla\omega_{0,j}^n\|_{L^{r_j}} 
\\ \label{FE} 
&\gtrsim 
\| \nabla\beta_j^n \cdot \nabla^\perp\eta_j^2(t^*)\|_{L^{r_j}} 
- 
\|\nabla\omega_{0,j} \cdot \nabla^\perp\eta_j^2(t^*)\|_{L^{r_j}} 
- 
\theta_n \|\nabla\omega_{0,j}^n\|_{L^{r_j}} 
\end{align} 
for any $j \geq 1$. 
Finally, observe that by \eqref{B2:eq:assump}
and the embedding $\dot B^1_{r_j,q_j}\subset \dot B^1_{r_j,2}\subset \dot W^{1,r_j}$, 
we have 
$$ 
\|\nabla\omega_{0,j} \cdot \nabla^\perp\eta_j^2(t^*)\|_{L^{r_j}} 
\lesssim 
\|\omega_j(t^*)\|_{B^1_{r_j,q_j}} 
\lesssim 
M_j^{1/3} 
$$ 
and by Lemma \ref{LLem} for any sufficiently large $n \gg 1$ we also have 
$$ 
\|\nabla\beta_j^n \cdot \nabla^\perp\eta_j^2(t^*)\|_{L^{r_j}} 
\gtrsim 
\| \partial_1\beta_j^n\partial_2\eta_j^2(t^*)\|_{L^{r_j}} 
- 
\|\partial_2\beta_j^n\partial_1\eta_j^2(t^*)\|_{L^{r_j}} 
\gtrsim 
M_j. 
$$ 
This establishes the second of the inequalities in \eqref{vort}. 
The desired sequence of velocities $\tilde u_j$ can now be obtained by selecting for each $j \geq 1$ 
a suitably large integer $n_j$ and setting $\tilde u_j =\nabla^\perp\Delta^{-1}\omega^{n_j}_j$. 
The proof of Theorem \ref{LIP-Besov} is completed. 
\end{proof} 
% 
%%%%%%%%%%%%%%%%%%%%%%%%%%%%%%%%%%% 
\section{Appendix: Continuity of the solution map in $W^{s,p}(\mathbb{R}^2)$} 
\label{Sobolev} 

In this section we mention the continuity of the solution map more precisely.
Continuous dependence of the solution map of the Euler equations with initial data 
in the Sobolev space 
%$H^s$ (for $s>2$) and, more generally,
% in
 $W^{s,p}$ for $p\geq 2$ and $s > 2/p + 2$ 
is of course well known 
(cf. e.g., Ebin and Marsden \cite{EbMa}, Kato and Lai \cite{KL} and Kato and Ponce \cite{KP-Duke}). 

\begin{theorem} \label{wellposed} 
The incompressible Euler equations \eqref{E} are locally well-posed 
in the sense of Hadamard in 
%the Sobolev space $W^{s,p}(\mathbb{R}^2)$ with $p\geq 2$ and $s>2/p+2$.
% 
\begin{enumerate} 
%\item[(i)] 
%the "little H\"older" space $c^{1,\sigma}(\mathbb{R}^2)$ with $0<\sigma<1$,  
\item[(i)] 
the Sobolev space $W^{s,p}(\mathbb{R}^2)$ with $p\geq 2$ and $s>2/p+2$,
\item[(ii)]
the Sobolev space $H^s(\mathbb{R}^2)$ with $s>2$.
\end{enumerate} 
\end{theorem} 

Note that the proof of (ii) is similar to that of (i). The key is to just use the new commutator estimate in \cite{FMRR}.
%(see also Remark \ref{external force case} in this paper).
%In the case of Sobolev spaces $W^{s,p}$, the issue was in 
% the commutator estimate in \cite{KP-Duke} (p.490, line 15).
%We found a counterexample of the commutator estimate.
% and its discovery motivated us 
%to consider the 2D-Euler flow in $W^{s,p}$ $(1+2/p<s\leq 2, p>2)$.
% which break down at low regularity range $1+2/p<s<2$.....??? -FIX} 
%On the other hand,
By using the new commutator estimate,
% which is appeared in 2014 (\cite{FMRR}),
 continuity of the solution map (in the 2D case) is restored in $p=2$ and $s>2$.
This is  related to  Parseval's identity.
% (for the Sobolev case). 

In what follows we shall sketch the proof of continuity of the solution map of \eqref{E} 
in $W^{s,p}$ for $s>2 + 2/p$ and $p>2$.
The problem turns out to be rather subtle. 
As Kato and Lai point out in \cite{KL} the first such result for the Euler equations in the Sobolev $H^s$ setting 
was proved in \cite{EbMa} for bounded domains and integer values $s > 1 + n/2$. 
The general case of $W^{s,p}$ was settled in \cite{KP-Duke} for unbounded domains and fractional 
$s > 2 + n/p$. 
Alternative proofs were also developed in \cite{K-Q}, \cite{KL} or \cite{BdV}. 

%%%%%%%%%%%%%
\subsection{Proof of Theorem \ref{wellposed}\,(i)} 
The argument follows closely that given by Kato and Ponce in \cite{KP-Duke}, Sect. 2 and 3. 
Only a minor adjustment is needed to one of the lemmas in their paper (see Lemma \ref{transport} below) 
which we restate here with a proof. 
%\textcolor{red}{Can we give a counter-ex. to Lem.1.1 in KP-Duke?} 

Let $D^s = (-\Delta)^{s/2}$ denote the fractional Laplacian as before. 
Recall the classical commutator estimate 
\begin{lem} \label{lem-com} 
If $s>0$ and $1<p<\infty$, then 
\begin{equation*}
\| D^s(fg)-f D^sg\|_p 
\lesssim 
\|\nabla f\|_\infty \| D^{s-1}g\|_p + \|D^s f\|_p \|g\|_\infty 
\end{equation*} 
for any $f$ and $g \in \mathcal{S}(\mathbb{R}^2)$. 
\end{lem}
\begin{proof} 
See \cite{KP-commutator}; Lem. X1. 
\end{proof} 

\begin{remark} (\cite{FMRR})
If $s>1$, then
\begin{equation}\label{Fefferman}
\| D^s((f\cdot\nabla)g)-(f\cdot\nabla) D^sg\|_2 
\lesssim 
\|f\|_{H^{s+1}} \|g\|_{H^s}. 
\end{equation} 
This is an improvement of the above lemma ($p=2$).
In their proof they essentially used Parseval's identity, so, we cannot directly 
generalize it to the $p>2$ case.
\end{remark}

We have the following\footnote{Cf. \cite{KP-Duke}; Lem. 1.1.} 
\begin{lem} \label{transport} 
Assume $1<p<\infty$ and $s_*>1+2/p$. 
Let $a \in C([0,T), W^{s_* p}(\mathbb{R}^2))$ be a divergence free vector field on $\mathbb{R}^2$. 
If $y_0 \in W^{s_*,p}(\mathbb{R}^2)$ then there exists a unique solution of the Cauchy problem 
\begin{equation} \label{transport equation}
\begin{cases} 
\partial_t y + a \cdot \nabla y = 0, 
\quad x\in\mathbb{R}^2 
\\
y(0, x) = y_0(x) 
\end{cases} 
\end{equation}
such that 
\begin{equation*}
\|y(t)\|_{W^{s_*,p}} \lesssim \|y_0\|_{W^{s_*,p}} \exp \left( C\int_0^t \|a(\tau)\|_{W^{s_*,p}} d\tau \right) 
\end{equation*} 
for any $0 \leq t < T$. 
\end{lem}
\begin{proof}[Proof of Lemma \ref{transport}] 
Since $s_* > 1 +2/p$ by Sobolev lemma there exists a smooth flow $\xi(t)$ of volume-preserving diffeomorphisms 
of class $C^1$ with $y(t,\xi(t,x)) = y_0(x)$ so that taking $L^p$ norms and changing variables we have 
\begin{equation} \label{tLp} 
\| y(t)\|_{L^p} = \|y_0\|_{L^p}. 
\end{equation} 
Applying $D^{s_*}$ to both sides of the transport equation \eqref{transport equation} we find 
$$ 
\partial_t (D^{s_*}y) + a\cdot\nabla D^{s_*}y = -D^{s_*}(a\cdot\nabla y) + a\cdot\nabla D^{s_*}y. 
$$ 
Evaluating this equation along $\xi(t)$, integrating it with respect to $t$ and taking $L^p$ norms as before, 
we obtain 
\begin{align} \label{tDLp} 
\|D^{s_*}y \|_{L^p} 
&\leq 
\| D^{s_*}y_0\|_{L^p} 
+ 
\int_0^t \big\| D^{s_*}(a\cdot\nabla y) - a\cdot\nabla D^{s_*}y \big\|_{L^p} d\tau 
\\ \nonumber 
&\lesssim 
\| D^{s_*}y_0\|_{L^p} 
+ 
\int_0^t \Big( \| Da\|_\infty \|D^{s_* -1}\nabla y\|_{L^p} + \|D^{s_*}a\|_{L^p} \|\nabla y\|_\infty \Big) d\tau 
\\ \nonumber 
&\lesssim 
\| D^{s_*}y_0\|_{L^p} 
+ 
\int_0^t \|a\|_{W^{s_*,p}} \|y\|_{W^{s_*,p}} d\tau 
\end{align} 
where in the second and third line we used the Kato-Ponce estimate of Lemma \ref{lem-com} 
and the Sobolev embedding theorem, respectively. 
Combining the estimates in \eqref{tLp} and \eqref{tDLp} with Gronwall's inequality 
we obtain the required estimate. 
\end{proof} 

\begin{remark}\label{external force case}
For $s_*>1+2/p$ and $p>2$, with the external force case
\begin{equation} \label{transport equation}
\begin{cases} 
\partial_t y + a \cdot \nabla y = f(t), 
\quad x\in\mathbb{R}^2 
\\
y(0, x) = y_0(x), 
\end{cases} 
\end{equation}
we have 
\begin{equation*}
\|y(t)\|_p\lesssim \|y_0\|_p+\int_0^t\|f(\tau)\|_pd\tau
\end{equation*}
and
\begin{equation*}
\|D^{s_*}y(t)\|_{p} \lesssim \|D^{s_*}y_0\|_{p} +\int_0^t \left(\|D^{s_*}a(\tau)\|_{p}\|D^{s_*}y(\tau)\|_{p}+\|D^{s_*}f(\tau)\|_p \right)d\tau. 
\end{equation*} 
for any $0 \leq t < T$. 
For $p=2$ and $s_*>1$, we can get better estimate by using Remark \ref{Fefferman}.
We have 
\begin{equation*}
\|y(t)\|_{H^{s_*}}\lesssim \|y_0\|_{H^{s_*}}+\int_0^t\left(\|a(\tau)\|_{H^{s_*}}\|y(\tau)\|_{H^{s_*}}+\|f(\tau)\|_{H^{s_*}}\right)d\tau.
\end{equation*}

%For $p=\infty$, $0<\sigma<1$, by using \eqref{Bourgain-Li's commutator estimate},
%we have 
%\begin{equation*}
%\|y(t)\|_{C^{1,\sigma}}\lesssim \|y_0\|_{C^{1,\sigma}}+\int_0^t\left(\|a(\tau)\|_{C^{1,\sigma}}\|y(\tau)\|_{C^{1,\sigma}}+\|f(\tau)\|_{C^{1,\sigma}}\right)d\tau
%\end{equation*}

\end{remark}

The rest part of   the  proof of Theorem \ref{wellposed} is essentially the same as Section 2 and Section 3 in \cite{KP-Duke}. 
The point is to estimate $\omega^j(t)-\omega^{j'}(t)$ in $W^{s-1,p}$ using a sequence of initial data
$\{u_0^j\}_{j=1}^\infty$ ($u_0^j:=\varphi_j\ast u_0$, $\varphi_j(x)=2^{2j}\varphi(2^jx)$, $\varphi\in\mathcal S$)
converging to $u_0$ in $W^{s,p}$ (corresponding initial vorticity is $\omega^j_0:= \mathrm{rot}\, u_0^j$, and its solution is $\omega^j(t):=\mathrm{rot}\, u^j(t)$.
In this case we need to set $s_*=s-1$, and need separability of the function spaces.
We apply Remark \ref{external force case} with $s_*=s-1>2/p+1$ to  
\begin{equation*}
\partial_t\left(\omega^j-\omega^{j'}\right)
=(u^j\cdot\nabla)(\omega^j-\omega^{j'})
+((u^j-u^{j'})\cdot\nabla)\omega^{j'},
\end{equation*}
then we have 
\begin{eqnarray*}
\|\omega^j(t)-\omega^{j'}(t)\|_{s_*,p}
&\leq&
\|\omega_0^j-\omega^{j'}_0\|_{s_*,p}\\
& &
+
\int_0^t\bigg(\|u^j(\tau)\|_{s_*,p}\|\omega^j(\tau)-\omega^{j'}(\tau)\|_{s_*,p}\\
& &
+\|u^j(\tau)-u^{j'}(\tau)\|_{s_*,p}\|\omega^{j'}(\tau)\|_{s_*,p}\\
& &
+\|u^j(\tau)-u^{j'}(\tau)\|_{s_*-1,p}\|\omega^{j'}(\tau)\|_{s_*+1,p}
\bigg)d\tau.
\end{eqnarray*}
On the other hand, again we apply Remark \ref{external force case} to the usual Euler equation with
$p=D^{-2}\mathrm{div}\, (u\cdot\nabla)u$ and 
\begin{equation*}
\|D^{-2+s_*}\nabla \mathrm{div}\, (u\cdot\nabla)v\|_p\leq 
\|\nabla u\|_\infty\|D^{s_*}v\|_p+\|D^{s_*}u\|_p\|\nabla v\|_\infty,
\end{equation*}
 we have 
\begin{equation*}
\|u^j(t)-u^{j'}(t)\|_{s_*-1,p}\leq \|u^j_0-u^{j'}_0\|_{s_*-1,p}\quad\text{with}\quad
\sup_{j\geq 1, t\in[0,T]}\|u^j_0\|_{s_*+1,p}\leq C
\end{equation*}
for $t\in [0,T]$.
By the density of $u^j_0$ in $W^{s_*+1,p}$, we see
\begin{equation*}
\|u^j_0-u^{j'}_0\|_{s_*-1,p}\|\omega^{j'}\|_{s_*+1,p}\to 0\quad\text{as}\quad
j,j'\to \infty.
\end{equation*} 
Combining the above calculations, we can show the continuity of the solution map in $p>2$, $2/p+2<s<2$.
The case $p=2$, $s>2$ is parallel so we omit it (just replace the commutator estimate to \eqref{Fefferman}).

%%%%%%%%%%%%%%%%%%%%%%%%%

%%%%%%%%%%%%%%%%
%%%%%%%%%%%%%%%%%%%%%%%%%%%%%%
\vspace{0.5cm}
\noindent
{\bf Acknowledgments.}\ 
Part of this work was done while GM was the Ulam Chair Visitor at the University of Colorado, Boulder,
and TY was an associate professor in the Department of Mathematics at  Tokyo Institute of Technology.
TY was partially supported by JSPS KAKENHI Grant Number 25870004. 
We thank Professor Yasushi Taniuchi and Professor Alexander Shnirelman for many inspiring conversations 
that led to this paper. 
%Also, this paper was developed during a stay of TY as an associate professor of the Department of Mathematics at Tokyo Institute of Technology.

%%%%%%%%%%%%%%%%
%%%%%%%%%%%%%%%%%%%%%%%%%%%%%%%%%%
\bibliographystyle{amsplain}

\end{document}